\documentclass{amsart}
\usepackage{amssymb}
\usepackage{amsfonts}
\usepackage{hyperref}

\newtheorem{theorem}{Theorem}[section]
\theoremstyle{plain}

\newtheorem{conjecture}{Conjecture}
\newtheorem{corollary}{Corollary}[section]

\newtheorem{lemma}{Lemma}

\newtheorem{proposition}{Proposition}
\newtheorem{remark}{Remark}

\numberwithin{equation}{section}

\begin{document}
\title[Sharp power mean bounds for the hypergeometric mean]{The second type
of power mean bounds for the hypergeometric mean}
\author{Zhen-Hang Yang and Jing-Feng Tian}
\address{Zhen-Hang Yang \\
Zhejiang Society for Electric Power, Hangzhou, Zhejiang, 310014, P. R. China}
\email{\url{yzhkm@163.com}}
\address{Jing-Feng Tian \\
Department of Mathematics and Physics\\
North China Electric Power University \\
Yonghua Street 619, 071003, Baoding, P. R. China}
\email{\url{tianjf@ncepu.edu.cn}}

\begin{abstract}
In this paper, by using signs rule for a class of special power series and
for the Laplace integrals, we find the best power mean bounds for the
hypergeometric function (sharp at the right endpoint). Our results contain
many known conclusions.
\end{abstract}

\maketitle


\section{Introduction}

This paper is a companion to \cite{Yang-arXiv:2610.10581}, and the aim is to
deal with the second type of power mean bounds for the Gaussian
hypergeometric function defined by%
\begin{equation*}
F\left( a,b;c;x\right) =\sum_{n=0}^{\infty }\frac{\left( a\right) _{n}\left(
b\right) _{n}}{\left( c\right) _{n}}\frac{x^{n}}{n!}\text{, \ }x\in \left(
0,1\right) ,
\end{equation*}%
where $a,b,c\in \mathbb{R}$ with $-c\notin \mathbb{N}\cup \left\{ 0\right\} $%
, $\left( a\right) _{n}=\Gamma \left( n+a\right) /\Gamma \left( a\right) $,
here $\Gamma (x)$ is the gamma function. The meaning of the first and second
type of power mean bounds for the Gaussian hypergeometric function is
explained in Remark \ref{R-12Hm-Pm}.

The hypergeometric mean value of positive numbers $x_{k}$ ($1\leq k\leq n$)
was defined by Carlson \cite{Carlson-PAMS-16-1965} using the R-function (see
\cite{Carlson-JMAA-7-1963}). In particular, the Gaussian hypergeometric
function $\left[ F\left( -a,b;c;1-x\right) \right] ^{1/a}$ is a mean of $1$
and $x$ provided $c>b>0$ for $x\in \left( 0,1\right) $. When $a=0$, this
mean is regarded as its limit at $a=0$, that is,%
\begin{equation*}
\lim_{a\rightarrow 0}\left[ F\left( -a,b;c;1-x\right) \right] ^{1/a}=\exp %
\left[ -\sum_{n=1}^{\infty }\frac{\left( b\right) _{n}}{\left( c\right) _{n}}%
\frac{x^{n}}{n}\right] .
\end{equation*}
Carlson \cite[Theorem 1]{Carlson-PAMS-16-1966} provided power mean bounds
for the hypergeometric mean:
\begin{equation}
\left[ F\left( -a,b;c;1-x\right) \right] ^{1/a}>\left( <\right) M_{a}\left(
x;\frac{b}{c}\right) \text{ if }a<\left( >\right) 1  \label{CI-a<>1}
\end{equation}%
for $x\in \left( 0,1\right) $, where%
\begin{equation*}
M_{p}\left( x;w\right) =\left( wx^{p}+1-w\right) ^{1/p}\text{ if }p\neq 0%
\text{ and }M_{0}\left( w,x\right) =x^{w}
\end{equation*}%
is the power mean of order $p\in \mathbb{R}$ of $1$ and $x$ with the weight $%
w\in \left( 0,1\right) $. It is known that $p\mapsto M_{p}\left( x;w\right) $
is increasing on $\mathbb{R}$.

Since 2001, Barnard and Richards found several improvements of Carlson's
inequality (\ref{CI-a<>1}) in \cite[Theorem 1]{Barnard-CMFT-1-2001}, \cite[%
Theorem 1]{Richards-JMAA-308-2005}, \cite{Barnard-JIPAM-8-2007}. It is worth
mentioning that, Richards in 2005 \cite[Theorem 1]{Richards-JMAA-308-2005}
obtained a very nice result using his unique tool, which states that

\begin{theorem}[{\textbf{\protect\cite[Theorem 1]{Richards-JMAA-308-2005}}}]
\textbf{\label{TR}}Suppose $a\leq 1$, $b>0$\ and $c\geq \max \left\{
-a,b\right\} $. If $c\geq \max \left\{ 1-2a,2b\right\} $, then%
\begin{equation}
\left[ F\left( -a,b;c;x\right) \right] ^{1/a}\geq M_{p}\left( 1-x;\frac{b}{c}%
\right)  \label{Hm>Pm-p}
\end{equation}%
\ for all $x\in \left( 0,1\right) $ if and only if $p\leq p_{0}$, where%
\begin{equation}
p_{0}=\frac{a+c}{1+c}.  \label{p0}
\end{equation}%
If $c\leq \min \left\{ 1-2a,2b\right\} $, then%
\begin{equation}
\left[ F\left( -a,b;c;x\right) \right] ^{1/a}\leq M_{p}\left( 1-x;\frac{b}{c}%
\right)  \label{Hm<Pm-p}
\end{equation}%
for all $x\in \left( 0,1\right) ,$if and only if $p\geq p_{0}$. If $%
c=1-2a=2b $; or $c=\max \left\{ -a,b\right\} $; or $a=1$, then the equality
holds.
\end{theorem}

\begin{remark}
\label{R-12Hm-Pm}(i) Theorem \ref{TR} provides a type of sharp power mean
bounds $M_{p_{0}}\left( 1-x;b/c\right) $ for the hypergeometric mean $\left[
F\left( -a,b;c;x\right) \right] ^{1/a}$, which is sharp at $x=0$. For
convenience, we call them \textquotedblleft the first type of power mean
bounds\textquotedblright . The inequalities (\ref{Hm>Pm-p}) and (\ref%
{Hm<Pm-p}) for $p=p_{0}$ can be abbreviated as by \textquotedblleft the
first Hm$>$Pm inequality\textquotedblright\ and \textquotedblleft the first
Hm$<$Pm inequality\textquotedblright , respectively, or be collectively
called \textquotedblleft the first Hm-Pm inequality
(inequalities)\textquotedblright .

(ii) Similarly, if the hypergeometric mean $\left[ F\left( -a,b;c;x\right) %
\right] ^{1/a}$ has the sharp power mean bounds $M_{p_{0}^{\ast }}\left(
1-x;b/c\right) $ which is sharp at $x=1$, then we call them
\textquotedblleft the second type of power mean bounds\textquotedblright .
The inequalities (\ref{Hm>Pm-p}) and (\ref{Hm<Pm-p}) for $p=p_{0}^{\ast }$
can be abbreviated as by \textquotedblleft the second Hm${>}$Pm
inequality\textquotedblright\ and \textquotedblleft the second Hm${<}$Pm
inequality\textquotedblright , respectively, or be collectively called
\textquotedblleft the second Hm-Pm inequality
(inequalities)\textquotedblright .
\end{remark}

\begin{remark}
\label{R-p0*<1}To ensure $p_{0}^{\ast }$ exists, it is clearly required that
$c>\max \left\{ -a,b,b-a\right\} >0$. Thus,%
\begin{equation}
p_{0}^{\ast }=\frac{a\ln \left( 1-b/c\right) }{\ln \frac{\Gamma \left(
c\right) \Gamma \left( c+a-b\right) }{\Gamma \left( c+a\right) \Gamma \left(
c-b\right) }}\text{ if }a\neq 0\text{ and }p_{0}^{\ast }=\frac{\ln \left(
1-b/c\right) }{\psi \left( c-b\right) -\psi \left( c\right) }.  \label{p0*}
\end{equation}%
Carlson \cite[Eq. (2.7)]{Carlson-PAMS-16-1966} pointed out that $a\mapsto %
\left[ F\left( -a,b;c;x\right) \right] ^{1/a}$ is increasing on $\mathbb{R}$%
, then so is
\begin{equation*}
a\mapsto \frac{1}{a}\ln F\left( -a,b;c;1\right) =\frac{1}{a}\ln \frac{\Gamma
\left( c\right) \Gamma \left( c+a-b\right) }{\Gamma \left( c+a\right) \Gamma
\left( c-b\right) }.
\end{equation*}%
This implies the following two conclusions:

(i) $p_{0}^{\ast }$ is an increasing function of $a$ on $\mathbb{R}$, and
therefore, $0<p_{0}^{\ast }<\left( >\right) 1$ if $a<\left( >1\right) $.

(ii) If $a<1$, then%
\begin{equation}
\frac{1}{a}\ln \frac{\Gamma \left( c\right) \Gamma \left( c+a-b\right) }{%
\Gamma \left( c+a\right) \Gamma \left( c-b\right) }<\ln \left( 1-\frac{b}{c}%
\right) <0.  \label{1/alng<0}
\end{equation}
\end{remark}

Recently, Yang and Tian \cite{Yang-arXiv:2610.10581} found the necessary and
sufficient conditions for the first Hm-Pm inequalities to hold in all
possible cases, which records as follows.

\begin{theorem}
\label{T-YT}Suppose that $a<1$, $b>0$ and $c>\max \left\{ -a,b\right\} $.
Let $p_{0}$ and $p_{0}^{\ast }$ be defined by (\ref{p0}) and (\ref{p0*}),
respectively.

(i) If $a-b+c\geq p_{0}$, then the first Hm${>}$Pm inequality holds for $%
x\in \left( 0,1\right) $ if and only if $c\geq 2b$.

(ii) If $a-b+c<p_{0}$, then the first Hm${>}$Pm inequality holds for $x\in
\left( 0,1\right) $ if and only if $a-b+c>0$, $c\geq 2b$ and $p_{0}\leq
p_{0}^{\ast }$.

(iii) If $a-b+c\leq p_{0}$, then the first Hm${<}$Pm inequality holds for $%
x\in \left( 0,1\right) $ if and only if $c\leq 2b$.

(iv) If $a-b+c>p_{0}$, then the first Hm${<}$Pm inequality holds for $x\in
\left( 0,1\right) $ if and only if $c\leq 2b$ and $p_{0}\geq p_{0}^{\ast }$.
\end{theorem}

For the second Hm-Pm inequalities, in 2010 Barnard, Richards and Tiedeman
\cite{Barnard-JMI-4-2010} made a conjecture as follows.

\begin{conjecture}[{\protect\cite[Conjecture I]{Barnard-JMI-4-2010}}]
\label{Cj-BRT}Let $a\leq 1$, $b>0$\ and $c>b>0$. If $c\geq \max \left\{
1-2a,2b\right\} $, then the second Hm${<}$Pm inequality holds for all $x\in
\left( 0,1\right) $. If $c\leq \min \left\{ 1-2a,2b\right\} $, then the
second Hm${>}$Pm inequality holds for all $x\in \left( 0,1\right) $.
\end{conjecture}

Applying Pales' Comparison Theorem \cite{Pales-JMAA-131-1988}, the authors
verified Conjecture \ref{Cj-BRT} is valid for the special cases $c=2b=2$
(see \cite[Proposition 1]{Barnard-JMI-4-2010}).

\begin{theorem}[{\protect\cite[Proposition 1]{Barnard-JMI-4-2010}}]
\label{T-BRT}Suppose $a\in (-1,1)$. If $a>-1/2$, then for all $r\in (0,1)$,%
\begin{equation*}
M_{\lambda }\left( 1-r,\frac{1}{2}\right) <\left[ F\left( -a,1;2;r\right) %
\right] ^{1/a}<M_{\mu }\left( 1-r,\frac{1}{2}\right)
\end{equation*}%
if $\lambda \leq \left( a+2\right) /3$ (sharp) and $\mu \geq \left( a\ln
2\right) /\ln \left( 1+a\right) $ (sharp). If $a<-1/2$, then the above
double inequality holds for all $r\in (0,1)$ if $\lambda \leq \left( a\ln
2\right) /\ln \left( 1+a\right) $ (sharp) and $\mu \geq \left( a+2\right) /3$
(sharp).
\end{theorem}

Motivated by Conjecture \ref{Cj-BRT} and \cite[Proposition 1]%
{Barnard-JMI-4-2010}, the aim of this paper is to find the second power mean
bounds for the hypergeometric means.

The paper is organized as follows. Our main results are stated in Section 2,
where Theorems \ref{MT-3} and \ref{MT-4} are listed in Section 2.1, and
several examples, as applications of our main results, are listed Section
2.2. In Section 3, an important tool, the signs for \textquotedblleft
Negative-Positive type power series\textquotedblright\ is introduced, and
several lemmas which are needed to prove Theorems \ref{MT-3} and \ref{MT-4}
are listed and established. The proofs of Theorems \ref{MT-3} and \ref{MT-4}
are presented in Sections 4 and 5.

\section{Main results and applications}

\subsection{Main results}

Before stating main results, we introduce several notations, that are,%
\begin{eqnarray}
\delta \left( p\right) &=&a-b+c-p,  \label{dp} \\
\Delta \gamma _{1}\left( p\right) &=&-\left( c+3\right) p^{2}+2\left(
2a-2b+2c+ab\right) p-a\left( a+c\right) ,  \label{dg1p} \\
\gamma _{1}\left( p\right) &=&-\left( c+1\right) p^{2}+\left(
2a-2b+2c+3ab+abc\right) p  \label{g1p} \\
&&-a\left( b+1\right) \left( a+c\right) .  \notag
\end{eqnarray}

\begin{theorem}
\label{MT-3}Suppose that (i) $a<1$, $b>0$ and $c>\max \left\{ b,b-a\right\} $%
; (ii) $p_{0}$ and $p_{0}^{\ast }$ be respectively defined by (\ref{p0}) and
(\ref{p0*}) with $p_{0}^{\ast }>p_{0}$. If $\left( a,b,c\right) \in \left(
E_{1,1,\pm 1}\cup E_{1,\pm 1,-1}\right) \backslash E_{0,0,-1}$, where%
\begin{eqnarray*}
E_{1,1,\pm 1} &=&\left\{ \left( a,b,c\right) :\delta \left( p_{0}^{\ast
}\right) \geq 0\text{, }\Delta \gamma _{1}\left( p_{0}^{\ast }\right) \geq
0\right\} , \\
E_{1,\pm 1,-1} &=&\left\{ \left( a,b,c\right) :\delta \left( p_{0}^{\ast
}\right) >0\text{, }\gamma _{1}\left( p_{0}^{\ast }\right) \leq 0\right\} ,
\\
E_{0,0,-1} &=&\left\{ \left( a,b,c\right) :\delta \left( p_{0}^{\ast
}\right) =\Delta \gamma _{1}\left( p_{0}^{\ast }\right) =0,\gamma _{1}\left(
p_{0}^{\ast }\right) \leq 0\right\} ,
\end{eqnarray*}%
then the second Hm-Pm inequality%
\begin{equation}
\left[ F\left( -a,b;c;x\right) \right] ^{1/a}<M_{p_{0}^{\ast }}\left( 1-x;%
\frac{b}{c}\right)   \label{Hm<Pm-p0*}
\end{equation}%
holds for $x\in \left( 0,1\right) $.
\end{theorem}

\begin{theorem}
\label{MT-4}Suppose that (i) $a<1$, $b>0$ and $c>\max \left\{ b,b-a\right\} $%
; (ii) $p_{0}$ and $p_{0}^{\ast }$ be respectively defined by (\ref{p0}) and
(\ref{p0*}) with $p_{0}^{\ast }<p_{0}$. If $\left( a,b,c\right) \in \left(
E_{-1,-1,\pm 1}\cup E_{-1,\pm 1,1}\right) \backslash E_{0,0,1}$, where%
\begin{eqnarray*}
E_{-1,-1,\pm 1} &=&\left\{ \left( a,b,c\right) :\delta \left( p_{0}^{\ast
}\right) \leq 0\text{, }\Delta \gamma _{1}\left( p_{0}^{\ast }\right) \leq
0\right\} , \\
E_{-1,\pm 1,1} &=&\left\{ \left( a,b,c\right) :\delta \left( p_{0}^{\ast
}\right) <0\text{, }\gamma _{1}\left( p_{0}^{\ast }\right) \geq 0\right\} ,
\\
E_{0,0,1} &=&\left\{ \left( a,b,c\right) :\delta \left( p_{0}^{\ast }\right)
=\Delta \gamma _{1}\left( p_{0}^{\ast }\right) =0,\gamma _{1}\left(
p_{0}^{\ast }\right) \geq 0\right\} ,
\end{eqnarray*}%
then the second Hm-Pm inequality%
\begin{equation}
\left[ F\left( -a,b;c;x\right) \right] ^{1/a}>M_{p_{0}^{\ast }}\left( 1-x;%
\frac{b}{c}\right)  \label{Hm>Pm-p0*}
\end{equation}%
holds for $x\in \left( 0,1\right) $.
\end{theorem}

\begin{remark}
\label{R-p0*><p0}Usually, It is rather complicated to compare $p_{0}^{\ast }$
with $p_{0}$ by analytical technique. However, there provides a simple
discrimination method in \cite[Remark 2 and Theorem 2.1]%
{Yang-arXiv:2610.10581}:

If $c\geq \left( \leq \right) 2b$ and $\delta \left( p_{0}\right)
=a-b+c-p_{0}\geq \left( \leq \right) 0$, then%
\begin{equation*}
C_{p_{0}}=\left( 1-\frac{b}{c}\right) ^{-1/p_{0}}\left[ \frac{\Gamma \left(
c\right) \Gamma \left( c+a-b\right) }{\Gamma \left( c+a\right) \Gamma \left(
c-b\right) }\right] ^{1/a}>\left( <\right) 1,
\end{equation*}%
which implies $p_{0}^{\ast }\geq \left( \leq \right) p_{0}$.
\end{remark}

\subsection{Applications}

We list several examples to illustrate applications of our main results. The
following corllary is equivalent to Theorem \ref{T-BRT}, and we offers a new
proof of it below.

\begin{corollary}
\label{C1}Suppose $a\in (-1,1)$. If $a>\left( <\right) -1/2$, then the
inequality%
\begin{equation}
M_{p_{0}}\left( 1-x,\frac{1}{2}\right) <\left( >\right) \left[ F\left(
-a,1;2;x\right) \right] ^{1/a}<\left( >\right) M_{p_{0}^{\ast }}\left( 1-x,%
\frac{1}{2}\right)   \label{F-a,1,2<>}
\end{equation}%
holds for $x\in (0,1)$ with the best constants $p_{0}=\left( a+2\right) /3$
and $p_{0}^{\ast }=\left( a\ln 2\right) /\ln \left( 1+a\right) $.
\end{corollary}

To prove Corollary \ref{C1} by Theorems \ref{MT-3} and \ref{MT-4}, we need
to verify several signs of $p_{0}^{\ast }-p_{0}$, $\delta \left( p_{0}^{\ast
}\right) $, $\Delta \gamma _{1}\left( p_{0}^{\ast }\right) $ and $\gamma
_{1}\left( p_{0}^{\ast }\right) $.

\begin{lemma}
\label{L-b=1,c=2}Let $a\in \left( -1,1\right) $ and $\left( b,c\right)
=\left( 1,2\right) $. Then

(i) $p_{0}=\left( a+c\right) /\left( c+1\right) <\left( >\right) p_{0}^{\ast
}$ if $a>\left( <\right) -1/2$.

(ii) $\delta \left( p_{0}^{\ast }\right) =a-b+c-p_{0}^{\ast }>\left(
<\right) 0$ if $a>\left( <\right) -1/2$;

(ii) $\Delta \gamma _{1}\left( p_{0}^{\ast }\right) =-\left( c+3\right)
\left( p_{0}^{\ast }\right) ^{2}+2\left( 2a-2b+2c+ab\right) p_{0}^{\ast
}-a\left( a+c\right) >0$ if $a>-1/2$.

(iii) $\gamma _{1}\left( p_{0}^{\ast }\right) >0$ if $a\in \left(
-1,-1/2\right) $
\end{lemma}

\begin{proof}
(i) We have%
\begin{equation*}
p_{0}=\frac{a+c}{1+c}=\frac{a+2}{3}\text{ \ and \ }p_{0}^{\ast }=\frac{a\ln
\left( 1-b/c\right) }{\ln \frac{\Gamma \left( c\right) \Gamma \left(
c+a-b\right) }{\Gamma \left( c+a\right) \Gamma \left( c-b\right) }}=\frac{%
a\ln 2}{\ln \left( a+1\right) }.
\end{equation*}%
Then%
\begin{equation*}
\delta \left( p_{0}\right) =a-1+2-\frac{a+2}{3}=\frac{2a+1}{3}.
\end{equation*}%
Since $c-2b=0$ and $\delta \left( p_{0}\right) >\left( <\right) 0$ if $%
a>\left( <\right) -1/2$, by Remark \ref{R-p0*><p0}, $p_{0}^{\ast }>\left(
<\right) p_{0}\ $if $a>\left( <\right) -1/2$, which proves the first
assertion.

(ii) We have%
\begin{equation*}
\delta \left( p_{0}^{\ast }\right) =a-b+c-p_{0}^{\ast }=1+a-\frac{a\ln 2}{%
\ln \left( a+1\right) }=\frac{h_{1}\left( a\right) }{\ln \left( a+1\right) },
\end{equation*}%
where%
\begin{equation*}
h_{1}\left( a\right) =\left( 1+a\right) \ln \left( a+1\right) -a\ln 2.
\end{equation*}%
Differentiation yileds that $h_{1}^{\prime \prime }\left( a\right) =1/\left(
a+1\right) >0$. This in combination with the facts that $h_{1}\left(
-1/2\right) =h\left( 0\right) =0$ and $h_{1}\left( -1^{+}\right)
=h_{1}\left( 1\right) =\ln 2>0$ leads to that $h_{1}\left( a\right) >0$ for $%
a\in \left( -1,-1/2\right) \cup \left( 0,1\right) $ and $h_{1}\left(
a\right) <0$ for $a\in \left( -1/2,0\right) $. It thus follows that $\delta
\left( p_{0}^{\ast }\right) >0$ for $a\in \left( -1/2,1\right) $ and $\delta
\left( p_{0}^{\ast }\right) <0$ for $a\in \left( -1,-1/2\right) $.

(iii) We have%
\begin{equation*}
\Delta \gamma _{1}\left( p_{0}^{\ast }\right) =-\frac{a}{\ln ^{2}\left(
a+1\right) }h_{2}\left( a\right) ,
\end{equation*}%
where%
\begin{equation*}
h_{2}\left( a\right) =\left( a+2\right) \ln ^{2}\left( a+1\right) -2\left(
\ln 2\right) \left( 3a+2\right) \ln \left( a+1\right) +5a\ln ^{2}2.
\end{equation*}%
Differentiation yileds%
\begin{eqnarray*}
h_{2}^{\prime }\left( a\right)  &=&\ln ^{2}\left( a+1\right) -2\frac{\left(
3\ln 2-1\right) a+3\ln 2-2}{a+1}\ln \left( a+1\right)  \\
&&-2\frac{3a+2}{a+1}\left( \ln 2\right) +5\ln ^{2}2, \\
\frac{\left( a+1\right) ^{2}}{2\left( 3a+4\right) }h_{2}^{\prime \prime
}\left( a\right)  &=&\frac{a+2+a\ln \left( a+1\right) }{4+3a}-\ln
2=h_{3}\left( a\right) , \\
\left( 3a+4\right) ^{2}h_{3}^{\prime }\left( a\right)  &=&4\ln \left(
a+1\right) +\frac{3a^{2}+2a-2}{a+1}=h_{4}\left( a\right) , \\
h_{4}^{\prime }\left( a\right)  &=&\frac{\left( a+2\right) \left(
3a+4\right) }{\left( a+1\right) ^{2}}>0\text{ for }a\in \left( -1,1\right)
\text{.}
\end{eqnarray*}

Since $h_{4}\left( -1/2\right) =-4\ln 2-9/2<0$ and $h_{4}\left( 1\right)
=4\ln 2+3/2>0$, there is an $a_{3}\in \left( -1/2,1\right) $ such that $%
h_{4}\left( a\right) <0$ for $a\in \left( -1/2,a_{3}\right) $ and $%
h_{4}\left( a\right) >0$ for $a\in \left( a_{3},1\right) $. Therefore, $%
h_{3}\left( a\right) $ is decreasing on $\left( -1/2,a_{3}\right) $ and
increasing on $\left( a_{3},1\right) $, which together with $h_{3}\left(
-1/2\right) =\left( 3-4\ln 2\right) /5>0$ and $h_{3}\left( 1\right)
=-3\left( \ln 4-1\right) /7<0$ indicates that there is an $a_{2}\in \left(
-1/2,1\right) $ such that $h_{3}\left( a\right) >0$ for $a\in \left(
-1/2,a_{2}\right) $ and $h_{3}\left( a\right) <0$ for $a\in \left(
a_{2},1\right) $. This implies that $h_{2}^{\prime }\left( a\right) $ is
increasing on $\left( -1/2,a_{2}\right) $ and decreasing on $\left(
a_{2},1\right) $, which in conjection with $h_{2}^{\prime }\left(
-1/2\right) =4\left( 3\ln 2-2\right) \left( \ln 2\right) >0$ and $%
h_{2}^{\prime }\left( 1\right) =-2\ln 2<0$ indicates that there is an $%
a_{1}\in \left( -1/2,1\right) $ such that $h_{2}^{\prime }\left( a\right) >0$
for $a\in \left( -1/2,a_{1}\right) $ and $h_{2}^{\prime }\left( a\right) <0$
for $a\in \left( a_{1},1\right) $. Since $h_{2}\left( -1/2\right) =0$ and $%
h_{2}\left( 1\right) =-2\ln ^{2}2$, we see that there is an $a_{0}\in \left(
-1/2,1\right) $ such that $h_{2}\left( a\right) >0$ for $a\in \left(
-1/2,a_{0}\right) $ and $h_{2}\left( a\right) <0$ for $a\in \left(
a_{0},1\right) $. Due to $h_{2}\left( 0\right) =0$, we find that $a_{0}=0$.
We thus conclude that $\Delta \gamma _{1}\left( p_{0}^{\ast }\right) >0$ for
$a\in \left( -1/2,1\right) $.

(iv) We have%
\begin{equation*}
\gamma _{1}\left( p_{0}^{\ast }\right) =-\left( a+2-3p_{0}^{\ast }\right)
\left( 2a-p_{0}^{\ast }\right) =3\left( p_{0}-p_{0}^{\ast }\right) \left(
\frac{a\ln 2}{\ln \left( a+1\right) }-2a\right) .
\end{equation*}%
It has been shown in the first part of this lemma that $\left(
p_{0}-p_{0}^{\ast }\right) >0$ for $a\in \left( -1,-1/2\right) $; since $%
a/\ln \left( a+1\right) >0$ for all $a>-1$, it is readily seen that $\left(
a\ln 2\right) /\ln \left( a+1\right) -2a>0$ for $a\in \left( -1,-1/2\right) $%
. Then $\gamma _{1}\left( p_{0}^{\ast }\right) >0$ for $a\in \left(
-1,-1/2\right) $.

This completes the proof.
\end{proof}

\begin{proof}[Proof of Corollary \protect\ref{C1}]
The first inequality of \ref{F-a,1,2<>} was proved in \cite[Example 1]%
{Yang-arXiv:2610.10581}. We prove the second one.

(i) When $a\in \left( -1/2,1\right) $, by Lemma \ref{L-b=1,c=2}, $%
p_{0}^{\ast }>p_{0}$, $\delta \left( p_{0}^{\ast }\right) >0$, $\delta
\left( p_{0}^{\ast }\right) >0$. Using Theorem \ref{MT-3} gives%
\begin{equation*}
\left[ F\left( -a,1;2;x\right) \right] ^{1/a}<M_{p_{0}^{\ast }}\left( 1-x,%
\frac{1}{2}\right)
\end{equation*}%
for $x\in \left( 0,1\right) $.

(ii) When $a\in \left( -1,-1/2\right) $, by Lemma \ref{L-b=1,c=2}, $%
p_{0}^{\ast }<p_{0}$, $\delta \left( p_{0}^{\ast }\right) <0$, $\gamma
_{1}\left( p_{0}^{\ast }\right) >0$. Using Theorem \ref{MT-4} gives%
\begin{equation*}
M_{p_{0}^{\ast }}\left( 1-x,\frac{1}{2}\right) <\left[ F\left(
-a,1;2;x\right) \right] ^{1/a}
\end{equation*}%
for $x\in \left( 0,1\right) $.
\end{proof}

\begin{remark}
Letting $a=0$ in Corollary \ref{C1} gives%
\begin{equation}
M_{2/3}\left( 1-x,\frac{1}{2}\right) <\lim_{a\rightarrow 0}\left[ F\left(
-a,1;2;x\right) \right] ^{1/a}<M_{\ln 2}\left( 1-x,\frac{1}{2}\right)
\label{Hm-Mp0-a=0}
\end{equation}%
for $x\in \left( 0,1\right) $ with the best constants $p_{0}=2/3$ and $%
p_{0}^{\ast }=\ln 2$. Note that%
\begin{eqnarray*}
\lim_{a\rightarrow 0}\left[ F\left( -a,1;2;x\right) \right] ^{1/a} &=&\exp %
\left[ -\sum_{n=1}^{\infty }\frac{x^{n}}{n\left( n+1\right) }\right]  \\
&=&\exp \left[ -\frac{\left( 1-x\right) \ln \left( 1-x\right) }{x}-1\right]
=I\left( 1,1-x\right) ,
\end{eqnarray*}%
where $I\left( 1,1-x\right) $ is the exponential mean of $1$ and $\left(
1-x\right) \in \left( 0,1\right) $. Then the double inequality (\ref%
{Hm-Mp0-a=0}) can be written as%
\begin{equation*}
M_{2/3}\left( x,\frac{1}{2}\right) <I\left( 1,x\right) <M_{\ln 2}\left( x,%
\frac{1}{2}\right)
\end{equation*}%
for $x\in \left( 0,1\right) $, which was proven in \cite%
{Pittinger-UBPEFSMF-680-1980} (see also \cite{Stolarsky-AMM-87-1980} \cite%
{Yang-MIA-10-2007}).
\end{remark}

\begin{corollary}
\label{C2}Suppose $a\in (0,1/2]$. Then the inequality%
\begin{equation}
M_{p_{0}}\left( 1-x,a\right) <\left[ F\left( -a,a;1;x\right) \right]
^{1/a}<M_{p_{0}^{\ast }}\left( 1-x,a\right)   \label{F-a,a,1<>}
\end{equation}%
holds for $x\in (0,1)$ with the best constants%
\begin{equation*}
p_{0}=\frac{a+1}{2}\text{ \ and \ }p_{0}^{\ast }=-\frac{a\ln \left(
1-a\right) }{\ln \frac{\pi a}{\sin \left( \pi a\right) }}.
\end{equation*}
\end{corollary}

As usual, to use Theorems \ref{MT-3} and \ref{MT-4}, we need to verify
several signs of $p_{0}^{\ast }-p_{0}$, $\delta \left( p_{0}^{\ast }\right) $
and $\Delta \gamma _{1}\left( p_{0}^{\ast }\right) $.

\begin{lemma}
\label{L-b=a,c=1}Let $\left( b,c\right) =\left( a,1\right) $ with $a\in
(0,1/2]$. Then

(i) $p_{0}=\left( a+1\right) /2<p_{0}^{\ast }=-\left( a\ln \left( 1-a\right)
\right) /\ln \frac{\pi a}{\sin \left( \pi a\right) }$.

(ii) $\delta \left( p_{0}^{\ast }\right) =1-p_{0}^{\ast }>0$.

(ii) $\Delta \gamma _{1}\left( p_{0}^{\ast }\right) =-4\left( p_{0}^{\ast
}\right) ^{2}+2\left( a^{2}+2\right) p_{0}^{\ast }-a\left( a+1\right) >0$.
\end{lemma}

\begin{proof}
(i) Since $c-2b=1-2a\geq 0$ and%
\begin{equation*}
\delta \left( p_{0}\right) =a-a+1-\frac{a+1}{2}=\frac{1-2a}{2}\geq 0,
\end{equation*}%
by Remark \ref{R-p0*><p0}, $p_{0}^{\ast }>p_{0}$, which proves the first
assertion (i).

(ii) We have $\delta \left( p_{0}^{\ast }\right) =1-p_{0}^{\ast }>0$, which
follows from Remark \ref{R-p0*<1} (i).

(iii) We have%
\begin{equation*}
\Delta \gamma _{1}\left( p_{0}^{\ast }\right) =-4p_{0}^{\ast 2}+2\left(
a^{2}+2\right) p_{0}^{\ast }-a\left( a+1\right) =-4\left( p_{0}^{\ast
}-p_{1}\right) \left( p_{0}^{\ast }-p_{2}\right) ,
\end{equation*}%
where%
\begin{eqnarray*}
p_{1} &=&p_{1}\left( a\right) =\frac{a^{2}+2-\sqrt{a^{4}-4a+4}}{4}, \\
p_{2} &=&p_{2}\left( a\right) =\frac{a^{2}+2+\sqrt{a^{4}-4a+4}}{4}.
\end{eqnarray*}%
If we show that $p_{0}^{\ast }>p_{1}$ and $p_{0}^{\ast }<p_{2}$, then $%
\Delta \gamma _{1}\left( p_{0}^{\ast }\right) >0$, and the proof is
completed. Now, since%
\begin{equation*}
p_{1}^{\prime }\left( a\right) =\frac{1}{2}\frac{a\sqrt{a^{4}-4a+4}+1-a^{3}}{%
\sqrt{a^{4}-4a+4}}>0,
\end{equation*}%
we have $p_{1}\left( a\right) <p_{1}\left( 1/2\right) =\left( 9-\sqrt{33}%
\right) /16$. Applying (i) of this lemma yields%
\begin{equation*}
p_{0}^{\ast }>p_{0}=\frac{a+1}{2}>\frac{1}{2}.
\end{equation*}%
Thus,%
\begin{equation*}
p_{0}^{\ast }-p_{1}>\frac{1}{2}-\frac{9-\sqrt{33}}{16}=\frac{\sqrt{33}-1}{16}%
>0.
\end{equation*}%
Using the known inequality $\left( \sin x\right) /x<e^{-x^{2}/6}\,$for $x\in
\left( 0,\pi /2\right) $ \cite[p. 606]{Yang-JMI-7-2013} which implies%
\begin{equation*}
\ln \frac{\pi a}{\sin \left( \pi a\right) }>\frac{\left( \pi a\right) ^{2}}{6%
}\text{ \ for }a\in (0,1/2],
\end{equation*}
we have%
\begin{equation*}
p_{0}^{\ast }=-\frac{a\ln \left( 1-a\right) }{\ln \left( \pi a/\sin \left(
\pi a\right) \right) }<-\frac{a\ln \left( 1-a\right) }{\left( \pi a\right)
^{2}/6}=-\frac{6}{\pi ^{2}}\frac{\ln \left( 1-a\right) }{a}=:p_{3}\left(
a\right) ,
\end{equation*}%
which is increasing in $a$ on $\left( 0,1\right) $ since%
\begin{equation*}
p_{3}^{\prime }\left( a\right) =\frac{6}{\pi ^{2}}\frac{a+\left( 1-a\right)
\ln \left( 1-a\right) }{a^{2}\left( 1-a\right) }>0\text{ for }a\in \left(
0,1\right) .
\end{equation*}%
On the other hand, since%
\begin{eqnarray*}
p_{2}^{\prime }\left( a\right)  &=&\frac{1}{2}\frac{a\sqrt{a^{4}-4a+4}%
-\left( 1-a^{3}\right) }{\sqrt{a^{4}-4a+4}}, \\
\left( a\sqrt{a^{4}-4a+4}\right) ^{2}-\left( 1-a^{3}\right) ^{2}
&=&-2a^{3}+4a^{2}-1<0\text{ for }a\in (0,1/2],
\end{eqnarray*}%
we have $p_{2}^{\prime }\left( a\right) <0$ for $a\in (0,1/2]$. Thus, $%
p_{3}\left( a\right) -p_{2}\left( a\right) $ is increasing in $a$ on $(0,1/2]
$, and hence,%
\begin{equation*}
p_{3}\left( a\right) -p_{2}\left( a\right) <p_{3}\left( \frac{1}{2}\right)
-p_{2}\left( \frac{1}{2}\right) =\frac{12}{\pi ^{2}}\ln 2-\frac{9+\sqrt{33}}{%
16}=-0.078...<0
\end{equation*}%
for $a\in (0,1/2]$. It then follows that
\begin{equation*}
p_{0}^{\ast }-p_{2}<p_{3}\left( a\right) -p_{2}\left( a\right) <0\text{ for }%
a\in (0,1/2].
\end{equation*}%
This completes the proof.
\end{proof}

We are now prove Corollary \ref{C2}.

\begin{proof}[Proof of Corollary \protect\ref{C2}]
(i) We first prove the first inequality. As noted in the proof of Lemma \ref%
{L-b=a,c=1}, $c-2b=1-2a\geq 0$ and $\delta \left( p_{0}\right) =1/2-a\geq 0$%
. By Theorem \ref{T-YT} (i), the required first inequality follows.

(ii) By Lemma \ref{L-b=a,c=1}, $p_{0}^{\ast }>p_{0}$, $\delta \left(
p_{0}^{\ast }\right) >0$ and $\Delta \gamma _{1}\left( p_{0}^{\ast }\right)
>0$. Using Theorem \ref{MT-3} gives the required second inequality.
\end{proof}

\begin{remark}
The complete $p$-elliptic integrals of the second is defined by%
\begin{equation*}
E_{p}\left( r\right) =\int_{0}^{1}\left( \frac{1-r^{p}t^{p}}{1-t^{p}}\right)
^{1/p}dt=\frac{\pi _{p}}{2}F\left( -\frac{1}{p},\frac{1}{p};1;r^{p}\right) ,
\end{equation*}%
where $\pi _{p}=B\left( 1/p,1-1/p\right) /p$ (see \cite{Takeuchi-KMJ-39-2016}%
). Let $a=1/p$ and $r^{\prime }=\left( 1-r^{p}\right) ^{1/p}$. The double
inequality (\ref{F-a,a,1<>}) can be written as%
\begin{equation}
\frac{\pi _{p}}{2}M_{\tau _{0}}\left( r^{\prime },\frac{1}{p}\right) <\frac{%
\pi _{p}}{2}F\left( -\frac{1}{p},\frac{1}{p};1;r^{p}\right) <\frac{\pi _{p}}{%
2}M_{\tau _{0}^{\ast }}\left( r^{\prime },\frac{1}{p}\right)   \label{p-e2<>}
\end{equation}%
holds for $p\geq 2$ and $r^{\prime }\in \left( 0,1\right) $ with the best
constants%
\begin{equation*}
\tau _{0}=\frac{p+1}{2}\text{ \ and \ }\tau _{0}^{\ast }=\frac{\ln \left(
p/\left( p-1\right) \right) }{\ln \left( \pi _{p}/2\right) }.
\end{equation*}%
The double inequality (\ref{p-e2<>}) was established in \cite[Corollary 1.3]%
{Zhao-BMSS-46-2023}.
\end{remark}

\begin{remark}
Taking $a=1/2$ in Corollary \ref{C2} gives%
\begin{equation*}
M_{p_{0}}\left( 1-x,\frac{1}{2}\right) <\left[ F\left( -\frac{1}{2},\frac{1}{%
2};1;x\right) \right] ^{2}<M_{p_{0}^{\ast }}\left( 1-x,\frac{1}{2}\right)
\end{equation*}%
for $x\in \left( 0,1\right) $ with the best constants $p_{0}=3/4$ and $%
p_{0}^{\ast }=\left( \ln 2\right) /\left( 2\ln \left( \pi /2\right) \right)
=0.767$. The above double inequality is equivalent to%
\begin{equation*}
M_{2p_{0}}\left( r^{\prime },\frac{1}{2}\right) <\frac{2}{\pi }E\left(
r\right) <M_{2p_{0}^{\ast }}\left( r^{\prime },\frac{1}{2}\right) \text{ \ }%
r^{\prime }=\sqrt{1-r^{2}},
\end{equation*}%
\ where the first one was proved by Qiu and Shen \cite[Theorem 2]%
{Qiu-JHIEE-3-1997} in 1997, while the second one was due to Qiu \cite[%
Corollary (1)]{Qiu-JHIEE-20(1)-2000} in 2000 (see also \cite%
{Alzer-JCAM-172-2004}). A skillful proof of the first one was presented by
Barnard, Pearce and Richards \cite[Theorem 1.1]{Barnard-JMA-31-2000} (see
also \cite{Barnard-JMA-32-2000}, \cite{Barnard-JMAA-260-2001}).
\end{remark}

Taking $\left( a,b,c\right) =\left( -3/2,1/2,3\right) $ leads to%
\begin{eqnarray*}
p_{0}^{\ast } &=&\frac{3\ln \left( 6/5\right) }{2\ln \left( 16/\left( 3\pi
\right) \right) }=0.516...>p_{0}=\frac{3}{8}\text{,} \\
\delta \left( p_{0}^{\ast }\right)  &=&1-p_{0}^{\ast }>0, \\
\Delta \gamma _{1}\left( p_{0}^{\ast }\right)  &=&-6\left( p_{0}^{\ast
}\right) ^{2}+\frac{5}{2}p_{0}^{\ast }+\frac{9}{4}=1.939...>0.
\end{eqnarray*}%
Using Theorem \ref{MT-3} we have

\begin{corollary}
The inequality%
\begin{equation*}
\left[ F\left( \frac{3}{2},\frac{1}{2};3;x\right) \right] ^{-2/3}<M_{p_{0}^{%
\ast }}\left( 1-x,\frac{1}{6}\right)
\end{equation*}%
holds for $x\in \left( 0,1\right) $ with the best constants $p_{0}^{\ast
}=0.516...$.
\end{corollary}

Taking $\left( a,b,c\right) =\left( -1/2,9/4,3\right) $ yields%
\begin{eqnarray*}
p_{0}^{\ast } &=&\frac{\ln 2}{\ln \frac{\Gamma \left( 3\right) \Gamma \left(
1/4\right) }{\Gamma \left( 5/2\right) \Gamma \left( 3/4\right) }}=0.464...<%
\frac{5}{8}=p_{0}, \\
\delta \left( p_{0}^{\ast }\right) &=&\frac{1}{4}-p_{0}^{\ast }=-0.214...<0,
\\
\gamma _{1}\left( p_{0}^{\ast }\right) &=&-4\left( p_{0}^{\ast }\right) ^{2}-%
\frac{25}{4}p_{0}^{\ast }+\frac{65}{16}=0.299...>0.
\end{eqnarray*}%
By Theorem \ref{MT-4} we obtain

\begin{corollary}
The inequality%
\begin{equation*}
\left[ F\left( \frac{1}{2},\frac{9}{4};3;x\right) \right] ^{-2}>M_{p_{0}^{%
\ast }}\left( 1-x,\frac{3}{4}\right)
\end{equation*}%
holds for $x\in \left( 0,1\right) $ with the best constants $p_{0}^{\ast
}=0.464...$.
\end{corollary}

\begin{remark}
\label{R-C3,4}In the above two corollaries, $\min \left\{ 2b,1-2a\right\}
<c<\max \left\{ 2b,1-2a\right\} $. Even if Conjecture \ref{Cj-BRT} is proved
to be valid, it still does not work.
\end{remark}

\section{Preliminaries}

\subsection{Two tools}

In order to prove our main results, we need two tools. The first tool is the
NP and PN-type power series and their signs rule. Recall that if every term
of a real sequence is nonnegative (nonpositive) and at least one is
non-zero, then this sequence is called a \textbf{positive (negative) sequence%
}. A real sequence $\left\{ a_{n}\right\} _{n\geq 0}$ is called an
negative-positive-type sequence, NP-type sequence for short, if there sexist
$m\in \mathbb{N}$ such that the subsequences $\left\{ a_{n}\right\} _{0\leq
n\leq m}$ and $\left\{ a_{n}\right\} _{n>m}$ are negative and positive
sequences, respectively. $\left\{ -a_{n}\right\} _{n\geq 0}$ is called a
positive-negative-type sequence, PN-type sequence for short.
Correspondingly, the normal power series $S\left( t\right)
=\sum_{k=0}^{\infty }a_{k}t^{k}$ is said to be NP-type, if the real sequence
$\left\{ a_{n}\right\} _{n\geq 0}$ is NP-type. $-S\left( t\right) $ is said
to be a PN-type power series (see \cite{Yang-RACSAM-115-2021-8}).

For the NP and PN-type power series, we have a simple but efficient signs
rule, which was first proved in \cite{Yang-arXiv-1705.05703} (see also \cite%
{Yang-AADM-13-2019}, \cite{Yang-JMAA-462-2018}).

\begin{proposition}
\label{P-sgnS}Let $S\left( x\right) $ be an NP-type power series converging
on the interval $\left( 0,r\right) $ ($r>0$). (i) If $S\left( r^{-}\right)
\leq 0$, then $S\left( x\right) <0$ for all $x\in \left( 0,r\right) $. (ii)
If $S\left( r^{-}\right) >0$, then there is a unique $x_{0}\in \left(
0,r\right) $ such that $S\left( x\right) <0$ for $x\in \left( 0,x_{0}\right)
$ and $S\left( x\right) >0$ for $x\in \left( x_{0},r\right) $.
\end{proposition}

\begin{remark}
\label{R-S><0}For an NP-type power series, it follows from Proposition \ref%
{P-sgnS} that if there is a $x_{1}\in \left( 0,r\right) $ such that $S\left(
x_{1}\right) <0$, then we have $S\left( x\right) <0$ for all $x\in \left(
0,x_{1}\right) $; if there is a $x_{2}\in \left( 0,r\right) $ such that $%
S\left( x_{2}\right) >0$, then we have $S\left( x\right) >0$ for all $x\in
\left( x_{2},r\right) $. For a PN-type power series $-S\left( x\right) $,
the signs of $-S\left( x\right) $ on $\left( 0,x_{1}\right) $ and $\left(
x_{2},r\right) $ are reversed.
\end{remark}

Proposition \ref{P-sgnS} has been applied in the study for certain special
functions, see for example, \cite{Belzunce-I:M-E-40-2007}, \cite%
{Yang-JIA-157-2015}, \cite{Yang-MIA-21-2018-2}, \cite{Yang-MIA-21-2018-4},
\cite{Yang-JMAA-467-2018}, \cite{Yang-AMC-348-2019}, \cite%
{Yang-JMAA-470-2019}, \cite{Yang-AAA-2014-702718}, \cite{Yang-JIA-2017-210},
\cite{Yang-JMI-12-2018}.

The second tool is a recurrence relation involving the hypergeometric
functions, which first appeared in \cite[Corollary 2]{Yang-arxiv-2204.04709}%
. Several special cases of Proposition \ref{P-PH-rr} can be seen in \cite%
{Yang-JMAA-467-2018}, \cite{Yang-RACSAM-115-2021}, \cite%
{Zhao-RACSAM-117-2023}.

\begin{proposition}[{\protect\cite[Corollary 2]{Yang-arxiv-2204.04709}}]
\label{P-PH-rr}Let $\left( a-1\right) \left( c-a\right) \neq 0$ and $%
-c\notin \mathbb{N\cup }\left\{ 0\right\} $. Then for $p\in \mathbb{R}$, we
have%
\begin{equation}
U\left( x\right) =\left( 1-x\right) ^{p}F\left( a,b;c;x\right)
=\sum_{n=0}^{\infty }u_{n}x^{n},  \label{U}
\end{equation}%
with $u_{0}=1$, $u_{1}=ab/c-p$ and for $n\geq 1$,%
\begin{eqnarray*}
u_{n+1} &=&\frac{2n^{2}+\left( a+b+c-2p-1\right) n+ab-cp}{\left( n+1\right)
\left( n+c\right) }u_{n} \\
&&-\frac{\left( n+a-p-1\right) \left( n+b-p-1\right) }{\left( n+1\right)
\left( n+c\right) }u_{n-1}.
\end{eqnarray*}
\end{proposition}

\subsection{Recurrence relations of the coefficients in the power series of $%
g_{p}\left( x\right) $}

Let%
\begin{equation}
f_{p}\left( x\right) =\frac{\left[ F\left( -a,b;c;x\right) \right] ^{1/a}}{%
\left[ 1-b/c+\left( b/c\right) \left( 1-x\right) ^{p}\right] ^{1/p}}.
\label{fp}
\end{equation}%
Logarithmic differentiation for (\ref{fp}) yields%
\begin{eqnarray}
\frac{f_{p}^{\prime }\left( x\right) }{f_{p}\left( x\right) } &=&\frac{1}{a}%
\frac{\left( -ab/c\right) F\left( 1-a,b+1;c+1;x\right) }{F\left(
-a,b;c;x\right) }+\frac{1}{p}\frac{p\left( b/c\right) \left( 1-x\right)
^{p-1}}{\left( 1-b/c+\left( b/c\right) \left( 1-x\right) ^{p}\right) }
\notag \\
&=&\frac{b\left( c-b\right) }{c}\frac{g_{p}\left( x\right) }{\left(
c-b+b\left( 1-x\right) ^{p}\right) F\left( -a,b;c;x\right) },  \label{dfp}
\end{eqnarray}

where%
\begin{equation}
\begin{array}{rl}
g_{p}\left( x\right) & =\dfrac{c}{c-b}\left( 1-x\right) ^{p-1}F\left(
-a,b;c;x\right) \bigskip \\
& -\left( 1+\dfrac{b}{c-b}\left( 1-x\right) ^{p}\right) F\left(
1-a,b+1;c+1;x\right) .%
\end{array}
\label{gp}
\end{equation}%
By Gauss contiguous relations \cite[p. 558]{Abramowitz-HMFFGMT-1972}, $%
g_{p}\left( x\right) $ can be simplified as%
\begin{equation}
g_{p}\left( x\right) =\left( 1-x\right) ^{p-1}F\left( 1-a,b;c+1;x\right)
-F\left( 1-a,b+1;c+1;x\right) \text{.}  \label{gp-a1}
\end{equation}%
Employing Cauchy product formula, $g_{p}\left( x\right) $ has the following
power series representation:%
\begin{equation*}
g_{p}\left( x\right) =\sum_{n=0}^{\infty }\left( u_{n}-v_{n}\right)
x^{n}=\sum_{n=0}^{\infty }\xi _{n}x^{n},
\end{equation*}%
where $\xi _{n}\equiv \xi _{n}\left( p\right) =u_{n}-v_{n}$ with%
\begin{eqnarray}
u_{n} &\equiv &u_{n}\left( p\right) =\sum_{k=0}^{n}\frac{\left( 1-p\right)
_{n-k}}{\left( n-k\right) !}\frac{\left( 1-a\right) _{k}\left( b\right) _{k}%
}{k!\left( c+1\right) _{k}},  \label{un} \\
v_{n} &=&\frac{\left( 1-a\right) _{n}\left( b+1\right) _{n}}{n!\left(
c+1\right) _{n}}  \label{vn}
\end{eqnarray}%
for all $n\geq 0$.

\begin{remark}
It should be noted that, $u_{n},v_{n}>0$ for all $n\geq 0$ if $a<1$, $c>b>0$
and $p\in \left( 0,1\right) $.
\end{remark}

\begin{remark}
An application of Proposition \ref{P-PH-rr} shows that the sequence $\left\{
u_{n}\right\} _{n\geq 1}$ satisfies%
\begin{equation}
u_{n+1}=\alpha _{n}u_{n}-\beta _{n}u_{n-1}\text{ for }n\geq 1  \label{un-rr}
\end{equation}%
with $u_{0}=1$ and $u_{1}=\left( 1-a\right) b/\left( c+1\right) -p+1$, where%
\begin{eqnarray}
\alpha _{n} &=&\frac{2n^{2}+\left( c-a+b-2p+3\right) n+\left( 1-a\right)
b+\left( c+1\right) \left( 1-p\right) }{\left( n+1\right) \left(
n+c+1\right) },  \label{alphn} \\
\beta _{n} &=&\frac{\left( n+1-a-p\right) \left( n+b-p\right) }{\left(
n+1\right) \left( n+c+1\right) }.  \label{betan}
\end{eqnarray}%
It is clear that $\alpha _{n},\beta _{n}>0$ for $n\geq 1$ if $a<1,c>b>0$, $%
p\in \left( 0,1\right) $.
\end{remark}

Substituting $u_{n}=\xi _{n}+v_{n}$ into (\ref{un-rr}) gives%
\begin{equation*}
\xi _{n+1}+v_{n+1}=\alpha _{n}\left( \xi _{n}+v_{n}\right) -\beta _{n}\left(
\xi _{n-1}+v_{n-1}\right) ,
\end{equation*}%
which can be written as%
\begin{eqnarray}
\xi _{n+1}-\alpha _{n}\xi _{n}+\beta _{n}\xi _{n-1} &=&-v_{n+1}+\alpha
_{n}v_{n}-\beta _{n}v_{n-1}  \label{xin-rr-1} \\
&=&\left( -\frac{v_{n+1}}{v_{n-1}}+\alpha _{n}\frac{v_{n}}{v_{n-1}}-\beta
_{n}\right) v_{n-1}\text{.}  \notag
\end{eqnarray}%
It is easy to verify that%
\begin{equation}
\frac{v_{n+1}}{v_{n}}=\frac{\left( n+1-a\right) \left( n+b+1\right) }{\left(
n+1\right) \left( n+c+1\right) },  \label{vn+1/vn}
\end{equation}%
from which it is deduced that%
\begin{equation}
-\frac{v_{n+1}}{v_{n-1}}+\alpha _{n}\frac{v_{n}}{v_{n-1}}-\beta _{n}=\tau
_{n}\gamma _{n},  \label{vn-gn}
\end{equation}%
where%
\begin{eqnarray}
\gamma _{n} &=&\gamma _{n}\left( p\right) =n^{2}p\left( a-b+c-p\right)
+ab\left( p\left( c+1\right) -\left( c+a\right) \right)  \label{gamn} \\
&&-n\left[ cp^{2}-\left( a-b+c+2ab\right) p+a\left( a+c\right) \right] ,
\notag \\
\tau _{n} &=&\left[ n\left( n+1\right) \left( n+c\right) \left( n+c+1\right) %
\right] ^{-1}.  \label{tn}
\end{eqnarray}%
Then we obtain the recurrence relation of $\xi _{n}$ as follows.

\begin{lemma}
\label{L-xin-rr}Let $u_{n}$ and $v_{n}$ be defined by (\ref{un}) and (\ref%
{vn}), respectively. The sequence $\xi _{n}\equiv \xi _{n}\left( p\right)
=u_{n}-v_{n}$ satisfies the recurrence relation%
\begin{equation}
\xi _{n+1}-\alpha _{n}\xi _{n}+\beta _{n}\xi _{n-1}=\tau _{n}\gamma
_{n}v_{n-1}  \label{xin-rr-0}
\end{equation}%
where $\alpha _{n}$ and $\beta _{n}$ are given by (\ref{alphn}) and (\ref%
{betan}), respectively. Moreover, if $a<1,c>b>0$, $p\in \left( 0,1\right) $
then $\alpha _{n},\beta _{n}>0$ for $n\geq 1$.
\end{lemma}

To prove main results we need to write (\ref{xin-rr-0}) in another form.

\begin{lemma}
\label{L-xin-rr-a1}Let $u_{n}$ and $v_{n}$ be defined by (\ref{un}) and (\ref%
{vn}), respectively. The sequence $\xi _{n}\equiv \xi _{n}\left( p\right)
=u_{n}-v_{n}$ satisfies the recurrence relation%
\begin{equation}
\left( \xi _{n+1}-\theta _{n}\xi _{n}\right) -\lambda _{n}\left( \xi
_{n}-\theta _{n-1}\xi _{n-1}\right) =\tau _{n}\gamma _{n}u_{n-1},
\label{xin-rr-2}
\end{equation}%
where%
\begin{eqnarray}
\theta _{n} &=&\frac{v_{n+1}}{v_{n}}=\frac{\left( n+1-a\right) \left(
n+b+1\right) }{\left( n+1\right) \left( n+c+1\right) },  \label{thn} \\
\lambda _{n} &=&\lambda _{n}\left( p\right) =\alpha _{n}-\theta _{n}=\frac{%
n^{2}+\left( c-2p+1\right) n+\left( a+c-p-cp\right) }{\left( n+1\right)
\left( n+c+1\right) },  \label{ln}
\end{eqnarray}%
$\gamma _{n}=\gamma _{n}\left( p\right) $ and $\tau _{n}$ are given by (\ref%
{gamn}) and (\ref{tn}), respectively. Moreover, if $a<1$, $b>0$, $c>\max
\left\{ -a,b\right\} $, then $\theta _{n}>0$ for $n\geq 1$, $\lambda
_{n}\left( p\right) >0$ for $n\geq 2$.
\end{lemma}

\begin{proof}
We write the relation (\ref{xin-rr-1}) as
\begin{eqnarray*}
&&\left( \xi _{n+1}-\frac{v_{n+1}}{v_{n}}\xi _{n}\right) -\left( \alpha _{n}-%
\frac{v_{n+1}}{v_{n}}\right) \left( \xi _{n}-\frac{v_{n}}{v_{n-1}}\xi
_{n-1}\right) \\
&=&\left( \left( \alpha _{n}-\frac{v_{n+1}}{v_{n}}\right) \frac{v_{n}}{%
v_{n-1}}-\beta _{n}\right) \xi _{n-1}-v_{n+1}+\alpha _{n}v_{n}-\beta
_{n}v_{n-1}.
\end{eqnarray*}%
Substituting $\xi _{n-1}=u_{n-1}-v_{n-1}$ into the right hand side of the
above equality and simplifying yield%
\begin{eqnarray*}
&&\left( \xi _{n+1}-\frac{v_{n+1}}{v_{n}}\xi _{n}\right) -\left( \alpha _{n}-%
\frac{v_{n+1}}{v_{n}}\right) \left( \xi _{n}-\frac{v_{n}}{v_{n-1}}\xi
_{n-1}\right) \\
&=&\left( -\frac{v_{n+1}}{v_{n-1}}+\alpha _{n}\frac{v_{n}}{v_{n-1}}-\beta
_{n}\right) u_{n-1},
\end{eqnarray*}%
which, by (\ref{vn-gn}), implies (\ref{xin-rr-2}). The formula (\ref{thn})
follows from (\ref{vn+1/vn}), and (\ref{ln}) follows by an elementary
computation.

It is clear that $\theta _{n}>0$ for $n\geq 1$ provided $a<1,b,c>0$. The
numerator of the fraction in (\ref{ln}) can be written as%
\begin{equation*}
\left( n-2\right) n+2\left( 1-p\right) n+\left( a+c\right) +\left(
n-p\right) \left( c+1\right) ,
\end{equation*}%
which is positive for $n\geq 2$ if $a<1$, $b>0$, $c>\max \left\{
-a,b\right\} $, thereby completing the proof.
\end{proof}

In the proofs of Theorems \ref{MT-3} and \ref{MT-4}, we need that $\lambda
_{1}\left( p_{0}^{\ast }\right) >0$, which is not an obvious fact.

\begin{lemma}
\label{L-l1p0*>0}Let $a<1$, $b>0$, $c>\max \left\{ b,b-a\right\} $. Then%
\begin{equation*}
\lambda _{1}\left( p_{0}^{\ast }\right) =\frac{a+2c+2-\left( 3+c\right)
p_{0}^{\ast }}{2\left( c+2\right) }>0.
\end{equation*}
\end{lemma}

\begin{proof}
By (\ref{1/alng<0}), to prove $\lambda _{1}\left( p_{0}^{\ast }\right) >0$,
it suffices to%
\begin{equation*}
\mu \left( b\right) =\frac{1}{a}\ln \frac{\Gamma \left( c\right) \Gamma
\left( c+a-b\right) }{\Gamma \left( c+a\right) \Gamma \left( c-b\right) }-%
\frac{3+c}{a+2c+2}\ln \left( 1-b/c\right) <0,
\end{equation*}%
which, due to $\mu \left( 0\right) =0$, in turn suffices to prove that $%
b\mapsto \mu \left( b\right) $ is decreasing on $\left( 0,\min \left\{
c,c+a\right\} \right) $. Differentiation yields%
\begin{equation*}
\left( c-b\right) \mu ^{\prime }\left( b\right) =-\left( c-b\right) \frac{%
\psi \left( a+c-b\right) -\psi \left( c-b\right) }{a}+\frac{3+c}{a+2c+2}%
=:\mu _{1}\left( b\right) .
\end{equation*}%
It was proved in \cite[Corollary 1]{Tian-RACSAM-115-2021} (see also \cite%
{Yang-JMAA-455-2017}) that, for $0\leq p-q\neq 1$, the function $t\mapsto
\left( t+r\right) \left[ \psi \left( t+p\right) -\psi \left( t+q\right) %
\right] /\left( p-q\right) $ is increasing (decreasing) on $\left( -q,\infty
\right) $ if and only if $r\leq \min \left\{ q,\left( p+q-1\right)
/2\right\} $ ($r\geq \max \left\{ q,\left( p+q-1\right) /2\right\} $). Then
the function%
\begin{equation*}
t\mapsto t\frac{\psi \left( t+a\right) -\psi \left( t\right) }{a}
\end{equation*}%
is decreasing on $\left( -\min \left\{ 0,a\right\} ,\infty \right) $ if $a<1$%
. This implies that $b\mapsto \mu _{1}\left( b\right) $ is also decreasing
on $\left( 0,\min \left\{ c,c+a\right\} \right) $, which yields%
\begin{equation*}
\frac{1}{c}\mu _{1}\left( b\right) <\frac{1}{c}\mu _{1}\left( 0\right) =-%
\frac{\psi \left( a+c\right) -\psi \left( c\right) }{a}+\frac{3+c}{c\left(
a+2c+2\right) }=\mu _{2}\left( c\right) .
\end{equation*}

If we prove $\mu _{2}\left( c\right) <0$, then $\mu _{1}\left( b\right) <0$,
and then $\mu ^{\prime }\left( b\right) <0$, and the proof is completed.
Now, since%
\begin{eqnarray*}
\mu _{2}\left( c+1\right) -\mu _{2}\left( c\right) &=&\frac{\psi \left(
c+1\right) -\psi \left( c+1+a\right) }{a}+\frac{4+c}{\left( c+1\right)
\left( a+2c+4\right) } \\
&&-\left[ \frac{\psi \left( c\right) -\psi \left( c+a\right) }{a}+\frac{3+c}{%
c\left( a+2c+2\right) }\right]
\end{eqnarray*}%
\begin{equation*}
=\frac{2\left( c^{2}+1-a\right) \left( a+c\right) +2c^{2}+\left(
a^{2}+1\right) c+\left( 1-a\right) \left( 5c+8\right) }{c\left( c+1\right)
\left( a+c\right) \left( a+2c+2\right) \left( a+2c+4\right) }>0,
\end{equation*}%
we have%
\begin{equation*}
\mu _{2}\left( c\right) <\mu _{2}\left( c+1\right) <\cdot \cdot \cdot
<\lim_{n\rightarrow \infty }\mu _{2}\left( c+n\right) =0,
\end{equation*}%
which completes the proof.
\end{proof}

\subsection{The signs of $\protect\gamma _{n}\left( p\right) $}

The signs of $\gamma _{n}\left( p\right) $ are needed to prove Theorems \ref%
{MT-3} and \ref{MT-4}.

\begin{lemma}
\label{L-sgngn-p}Suppose that $a<1,c>b>0$, $p\in \left( 0,1\right) $ and let
$\gamma _{n}\equiv \gamma _{n}\left( p\right) $ be defined by (\ref{gamn}).

(i) If $\delta \left( p\right) \geq \left( \leq \right) 0$, $\Delta \gamma
_{1}\left( p\right) \geq \left( \leq \right) 0$ and $\gamma _{1}\left(
p\right) \geq \left( \leq \right) 0$, then $\gamma _{n}\left( p\right) \geq
\left( \leq \right) 0$ for all $n\geq 1$.

(ii) If $\delta \left( p\right) >\left( <\right) 0$ and $\gamma _{1}\left(
p\right) \leq \left( \geq \right) 0$, or $\delta \left( p\right) =0$, $%
\Delta \gamma _{1}\left( p\right) >\left( <\right) 0$ and $\gamma _{1}\left(
p\right) <\left( >\right) 0$, then there is an $n_{1}\geq 1$ such that $%
\gamma _{n}\left( p\right) <\left( >\right) 0$ for $1\leq n\leq n_{1}-1$ and
$\gamma _{n}\left( p\right) >\left( <\right) 0$ for $n\geq n_{1}+1$.
\end{lemma}

\begin{proof}
We have%
\begin{eqnarray*}
\Delta \gamma _{n}\left( p\right) &:&=\gamma _{n+1}\left( p\right) -\gamma
_{n}\left( p\right) =2np\left( a-b+c-p\right) \\
&&-\left( c+1\right) p^{2}+2\left( c+a-b+ab\right) p-a\left( a+c\right) , \\
\Delta \gamma _{1}\left( p\right) &=&-\left( c+3\right) p^{2}+2\left(
2a-2b+2c+ab\right) p-a\left( a+c\right) , \\
\Delta \gamma _{n+1}\left( p\right) -\Delta \gamma _{n}\left( p\right)
&=&2p\left( a-b+c-p\right) =2p\delta \left( p\right) .
\end{eqnarray*}

(i) If $\delta \left( p\right) \geq \left( \leq \right) 0$, then the
sequence $\left\{ \Delta \gamma _{n}\left( p\right) \right\} _{n\geq 1}$ is
increasing (decreasing). This together with $\Delta \gamma _{1}\left(
p\right) \geq \left( \leq \right) 0$ implies that $\left\{ \gamma _{n}\left(
p\right) \right\} _{n\geq 1}$ is also increasing (decreasing). Therefore, if
$\gamma _{1}\left( p\right) \geq \left( \leq \right) 0$, then $\gamma
_{n}\left( p\right) \geq \left( \leq \right) \gamma _{1}\left( p\right) \geq
\left( \leq \right) 0$ for all $n\geq 1$.

(ii) If $\delta \left( p\right) >\left( <\right) 0$, then the sequence $%
\left\{ \gamma _{n}\left( p\right) \right\} _{n\geq 1}$ is convex (concave).
This together with $\gamma _{1}\left( p\right) <\left( >\right) 0$ and $%
\lim_{n\rightarrow \infty }\left[ n^{-2}\gamma _{n}\left( p\right) \right]
=p\delta \left( p\right) >\left( <\right) 0$ yields that there is an $%
n_{1}\geq 1$ such that $\gamma _{n}\left( p\right) <\left( >\right) 0$ for $%
1\leq n\leq n_{1}-1$ and $\gamma _{n}\left( p\right) >\left( <\right) 0$ for
$n\geq n_{1}+1$.

If $\delta \left( p\right) =0$, then%
\begin{eqnarray*}
\gamma _{n}\left( p\right) &=&-n\left[ cp^{2}-\left( a-b+c+2ab\right)
p+a\left( a+c\right) \right] +ab\left( p\left( c+1\right) -\left( c+a\right)
\right) , \\
\Delta \gamma _{n}\left( p\right) &=&cp^{2}+\left( a-b+c+2ab\right)
p-a\left( a+c\right) \equiv \Delta \gamma _{1}\left( p\right) .
\end{eqnarray*}%
This shows that the sequence $\left\{ \gamma _{n}\left( p\right) \right\}
_{n\geq 1}$ is increasing (decreasing) if $\Delta \gamma _{1}\left( p\right)
>\left( <\right) 0$. Since $\gamma _{1}\left( p\right) <\left( >\right) 0$
and $\lim_{n\rightarrow \infty }\left[ n^{-1}\gamma _{n}\left( p\right) %
\right] =\Delta \gamma _{1}\left( p\right) >\left( <\right) 0$, there is an $%
n_{1}\geq 1$ such that $\gamma _{n}\left( p\right) <\left( >\right) 0$ for $%
1\leq n\leq n_{1}-1$ and $\gamma _{n}\left( p\right) >\left( <\right) 0$ for
$n\geq n_{1}+1$.

This completes the proof.
\end{proof}

\section{Proof of Theorem \protect\ref{MT-3}}

\begin{proof}[Proof of Theorem \protect\ref{MT-3}]
Due to $b>0$ and $c>\max \left\{ b,b-a\right\} $, so $c,c+a,c-b,c+a-b>0$,
which guarantees the existence of $p_{0}^{\ast }$. It is easy to see that,
for $p=p_{0}^{\ast }$,%
\begin{equation*}
f_{p_{0}^{\ast }}\left( 1\right) =\left( \dfrac{c}{c-b}\right)
^{1/p_{0}^{\ast }}\left[ \dfrac{\Gamma \left( c\right) \Gamma \left(
c+a-b\right) }{\Gamma \left( c+a\right) \Gamma \left( c-b\right) }\right]
^{1/a}=1=f_{p_{0}^{\ast }}\left( 0\right) .
\end{equation*}

A direct computation gives%
\begin{eqnarray*}
\xi _{0}\left( p\right) &=&0\text{, \ \ \ \ \ }\xi _{1}\left( p\right) =%
\frac{a+c}{c+1}-p, \\
\xi _{2}\left( p\right) &=&\frac{1}{2}p^{2}-\frac{1}{2}p\frac{2b+3c-2ab+3}{%
c+1}-\frac{\left( a+c\right) \left( a-b-c+ab-3\right) }{\left( c+1\right)
\left( c+2\right) }.
\end{eqnarray*}

Since $p_{0}<p_{0}^{\ast }<1$, we see that%
\begin{equation*}
\xi _{1}\left( p_{0}^{\ast }\right) =\frac{a+c}{c+1}-p_{0}^{\ast
}=p_{0}-p_{0}^{\ast }<0.
\end{equation*}%
If we prove that%
\begin{equation}
g_{p_{0}^{\ast }}\left( x\right) =\sum_{n=0}^{\infty }\xi _{n}\left(
p_{0}^{\ast }\right) x^{n}\text{ is an NP-type power series with }%
g_{p_{0}^{\ast }}\left( 1\right) >0,  \label{A-gp0*-NP}
\end{equation}%
then by Proposition \ref{P-sgnS} there is an $x_{0}\in \left( 0,1\right) $
such that $g_{p_{0}^{\ast }}\left( x\right) <0$ for $x\in \left(
0,x_{0}\right) $ and $g_{p_{0}^{\ast }}\left( x\right) >0$ for $x\in \left(
x_{0},1\right) $. This, by (\ref{dfp}), implies that $f_{p_{0}^{\ast
}}\left( x\right) $ is strictly decreasing on $\left( 0,x_{0}\right) $ and
increasing on $\left( x_{0},1\right) $. It then follows that%
\begin{equation*}
f_{p_{0}^{\ast }}\left( x_{0}\right) \leq f_{p_{0}^{\ast }}\left( x\right)
<\max \left\{ f_{p_{0}^{\ast }}\left( 0\right) ,f_{p_{0}^{\ast }}\left(
1\right) \right\} =1,
\end{equation*}%
which implies that the inequality (\ref{Hm<Pm-p0*}) holds for $x\in \left(
0,1\right) $.

Note that $\left( a,b,c\right) \in \left( E_{1,1,\pm 1}\cup E_{1,\pm
1,-1}\right) \backslash E_{0,0,-1}$ can be divided into the following four
cases%
\begin{equation*}
\begin{array}{cccc}
\text{Case 1:} & \delta \left( p_{0}^{\ast }\right) \geq 0, & \Delta \gamma
_{1}\left( p_{0}^{\ast }\right) \geq 0, & \gamma _{1}\left( p_{0}^{\ast
}\right) >0; \\
\text{Case 2:} & \delta \left( p_{0}^{\ast }\right) >0, & \Delta \gamma
_{1}\left( p_{0}^{\ast }\right) \geq 0, & \gamma _{1}\left( p_{0}^{\ast
}\right) \leq 0; \\
\text{Case 3:} & \delta \left( p_{0}^{\ast }\right) =0, & \Delta \gamma
_{1}\left( p_{0}^{\ast }\right) >0, & \gamma _{1}\left( p_{0}^{\ast }\right)
\leq 0; \\
\text{Case 4:} & \delta \left( p_{0}^{\ast }\right) >0 & \Delta \gamma
_{1}\left( p_{0}^{\ast }\right) <0 & \gamma _{1}\left( p_{0}^{\ast }\right)
\leq 0.%
\end{array}%
\end{equation*}%
We now distinguish the above four cases to prove the assertion (\ref%
{A-gp0*-NP}). For convenience, the sequences $H_{n}\left( p_{0}^{\ast
}\right) $ are denoted simply by $H_{n}$ in the proof below, where $H=\eta $%
, $\lambda $, $\tau $ $\gamma $, $u$, $\xi $.

\textbf{Case 1}: $\delta \left( p_{0}^{\ast }\right) \geq 0,\Delta \gamma
_{1}\left( p_{0}^{\ast }\right) \geq 0$ and $\gamma _{1}\left( p_{0}^{\ast
}\right) >0$. By Lemma \ref{L-sgngn-p} (i) we see that $\gamma _{n}\geq 0$
for all $n\geq 1$.

\textbf{First}, although $\xi _{1}\left( p_{0}^{\ast }\right) <0$, there
must be a positive term in the sequence $\{\xi _{n}\}_{n\geq 0}$. If not,
namely, there is no positive term, then $\xi _{n}<0$ for all $n\in \mathbb{N}
$. This implies $g_{p_{0}^{\ast }}\left( x\right) <0$ for all $x\in \left(
0,1\right) $, which, by (\ref{dfp}), yields that $f_{p_{0}^{\ast }}\left(
x\right) $ is decreasing on $\left( 0,1\right) $. It results in%
\begin{equation}
1=f_{p_{0}^{\ast }}\left( 1\right) <f_{p_{0}^{\ast }}\left( x\right)
<f_{p_{0}^{\ast }}\left( 0\right) =1,  \label{1<1}
\end{equation}%
which leads to a contradiction.

\textbf{Second}, assume that $\xi _{n_{0}}$ ($n_{0}\geq 2$) is the first
positive term of the sequence $\{\xi _{n}\}_{n\geq 0}$, that is, $\xi
_{n}\leq 0$ for $1\leq n\leq n_{0}-1$ and $\xi _{n_{0}}>0$. We claim that
all terms after that ($n\geq n_{0}+1$) are positive, which indicates that $%
\left\{ \xi _{n}\right\} _{n\geq 0}$ is an NP-type sequence. We prove now
that $\xi _{n}>0$ for all $n\geq n_{0}+1$ by induction.

From Lemma \ref{L-xin-rr-a1} it follows that%
\begin{equation}
\eta _{n+1}-\lambda _{n}\eta _{n}=\tau _{n}\gamma _{n}u_{n-1},
\label{etan-rr}
\end{equation}%
where%
\begin{equation*}
\eta _{n+1}=\xi _{n+1}-\theta _{n}\xi _{n},
\end{equation*}%
and $\theta _{n},\tau _{n},\gamma _{n}>0$ for $n\geq 1$ and $\lambda _{n}>0$
for $n\geq 2$. Then%
\begin{equation}
\eta _{n_{0}+1}=\xi _{n_{0}+1}-\theta _{n_{0}}\xi _{n_{0}}=\lambda
_{n_{0}}\left( \xi _{n_{0}}+\theta _{n_{0}-1}\left( -\xi _{n_{0}-1}\right)
\right) +\tau _{n_{0}}\gamma _{n_{0}}u_{n_{0}-1}>0,  \label{etan0+1>0}
\end{equation}%
which yields $\xi _{n_{0}+1}>\theta _{n_{0}}\xi _{n_{0}}>0$.

Suppose that $\xi _{k}\geq 0$ for $n_{0}+1\leq k\leq m$. We show that $\xi
_{k+1}\geq 0$ by induction. The recurrence relation (\ref{xin-rr-2}) yields%
\begin{equation}
\eta _{k+1}-\lambda _{k}\eta _{k}=\tau _{k}\gamma _{k}u_{k-1}\geq 0
\label{etan+1>}
\end{equation}%
for $k\geq n_{0}+1$. Taking into account (\ref{etan0+1>0}) and (\ref{etan+1>}%
) we obtain%
\begin{equation*}
\eta _{k+1}>\lambda _{k}\eta _{k}>\left( \lambda _{k}\lambda _{k-1}\right)
\eta _{k-1}>\cdot \cdot \cdot >\left( \lambda _{k}\lambda _{k-1}\cdot \cdot
\cdot \lambda _{n_{0}+1}\right) \eta _{n_{0}+1}>0,
\end{equation*}%
which yields $\eta _{k+1}=\xi _{k+1}-\theta _{k}\xi _{k}>0$. Using the
induction hypothesis we immediately deduce that $\xi _{k+1}>\theta _{k}\xi
_{k}>0$. By induction we acquire that $\xi _{n}>0$ for all $n\geq n_{0}$.

\textbf{Third}, we claim that $g_{p_{0}^{\ast }}\left( 1\right) >0$. If not,
that is, $g_{p_{0}^{\ast }}\left( 1\right) \leq 0$, then by Proposition \ref%
{P-sgnS}, $g_{p_{0}^{\ast }}\left( x\right) \leq 0$ for all $x\in \left(
0,1\right) $. This also leads to the same contradiction as (\ref{1<1}).
Consequently, $\left\{ \xi _{n}\right\} _{n\geq 0}$ is an NP-type sequence,
and the assertion (\ref{A-gp0*-NP}) is valid.

\textbf{Case 2:} $\delta \left( p_{0}^{\ast }\right) >0$, $\Delta \gamma
_{1}\left( p_{0}^{\ast }\right) \geq 0$ and $\gamma _{1}\left( p_{0}^{\ast
}\right) \leq 0$. By Lemma \ref{L-sgngn-p} (ii) there is an $n_{1}\geq 1$
such that $\gamma _{n}<0$ for $1\leq n\leq n_{1}-1$ and $\gamma _{n}>0$ for $%
n\geq n_{1}+1$. Then Lemma \ref{L-xin-rr-a1} reveals%
\begin{equation}
\eta _{n+1}-\lambda _{n}\eta _{n}=\tau _{n}\gamma _{n}u_{n-1}\left\{
\begin{array}{ll}
<0 & \text{if }1\leq n\leq n_{1}-1, \\
>0 & \text{if }n\geq n_{1}+1,%
\end{array}%
\right.  \label{etan-rr-+}
\end{equation}%
where $\theta _{n},\tau _{n},\gamma _{n}>0$ for $n\geq 1$ and $\lambda
_{n}>0 $ for $n\geq 2$. We first show that $\xi _{n}=\xi _{n}\left(
p_{0}^{\ast }\right) <0$ for $1\leq n\leq n_{1}$ by induction. Since $\xi
_{0}=0$, $\xi _{1}=p_{0}-p_{0}^{\ast }<0$, we have%
\begin{equation*}
\eta _{1}=\xi _{1}-\theta _{0}\xi _{0}=\xi _{1}<0\text{.}
\end{equation*}%
Moreover, from Lemma \ref{L-l1p0*>0} we see that%
\begin{equation*}
\lambda _{1}=\lambda _{1}\left( p_{0}^{\ast }\right) =\frac{a+2c+2-\left(
3+c\right) p_{0}^{\ast }}{2\left( c+2\right) }>0\text{.}
\end{equation*}%
Using the first relation of (\ref{etan-rr-+}) for $n=1$ yields%
\begin{equation*}
\eta _{2}=\xi _{2}-\theta _{1}\xi _{1}=\lambda _{1}\eta _{1}+\tau _{1}\gamma
_{1}u_{0}<0,
\end{equation*}%
which implies that $\xi _{2}<\theta _{1}\xi _{1}<0$. Suppose that $\xi
_{k}<0 $ for $1\leq k\leq n\leq n_{1}-1$. An application of the first
relation of (\ref{etan-rr-+}) gives%
\begin{equation*}
\eta _{k+1}<\lambda _{k}\eta _{k}<\lambda _{k}\lambda _{k-1}\eta
_{k-1}<\cdot \cdot \cdot <\left( \lambda _{k}\lambda _{k-1}\cdot \cdot \cdot
\lambda _{2}\right) \eta _{2}<0,
\end{equation*}%
which implies that $\xi _{k+1}-\theta _{k}\xi _{k}<0$, and therefore, $\xi
_{k+1}<\theta _{k}\xi _{k}<0$ by the induction hypothesis. It follows from
induction that $\xi _{n}<0$ for $1\leq n\leq n_{1}$.

Second, there must be a positive term in the sequence $\left\{ \xi
_{n}\right\} _{n\geq n_{1}+1}$. If not, namely, there is no positive term,
then $\xi _{n}<0$ for all $n\geq n_{1}+1$. Thus, $\xi _{n}<0$ for $n\in
\mathbb{N}$, and then, $g_{p_{0}^{\ast }}\left( x\right) =\sum_{n=0}^{\infty
}\xi _{n}x^{n}<0$ for all $x\in \left( 0,1\right) $, which, by (\ref{dfp}),
yields that $f_{p_{0}^{\ast }}\left( x\right) $ is decreasing on $\left(
0,1\right) $. It leads to the same contradiction as (\ref{1<1}). Now assume
that $\xi _{n_{0}}$ ($n_{0}\geq n_{1}+1$) is the first positive term of the
sequence $\{\xi _{n}\}_{n\geq n_{1}+1}$, that is, $\xi _{n}\leq 0$ for $%
n_{1}+1\leq n\leq n_{0}-1$ and $\xi _{n_{0}}>0$. We claim that all terms
after that ($n\geq n_{0}+1$) are positive, which indicates that $\left\{ \xi
_{n}\right\} _{n\geq 0}$ is an NP-type sequence. We now prove that $\xi
_{n}>0$ for all $n\geq n_{0}+1\geq n_{1}+2$ by induction. Using the second
relation of (\ref{etan-rr-+}) yields%
\begin{equation*}
\eta _{n_{0}+1}=\xi _{n_{0}+1}-\theta _{n_{0}}\xi _{n_{0}}=\lambda
_{n_{0}}\left( \xi _{n_{0}}+\theta _{n_{0}-1}\left( -\xi _{n_{0}-1}\right)
\right) +\tau _{n_{0}}\gamma _{n_{0}}u_{n_{0}-1}>0,
\end{equation*}%
which gives that $\xi _{n_{0}+1}>\theta _{n_{0}}\xi _{n_{0}}>0$. Suppose
that $\xi _{k}>0$ for $n_{0}+1\leq k\leq n$. Using the second relation of (%
\ref{etan-rr-+})\ again yields%
\begin{equation*}
\eta _{k+1}>\lambda _{k}\eta _{k}>\lambda _{k}\lambda _{k-1}\eta
_{k-1}>\cdot \cdot \cdot >\left( \lambda _{k}\lambda _{k-1}\cdot \cdot \cdot
\lambda _{n_{0}+1}\right) \eta _{n_{0}+1}>0,
\end{equation*}%
which leads to $\xi _{k+1}-\theta _{k}\xi _{k}>0$, and therefore, $\xi
_{k+1}>\theta _{k}\xi _{k}>0$. By induction we arrive at $\xi _{k}>0$ for
all $k\geq n_{0}$. Consequently, $\left\{ \xi _{n}\right\} _{n\geq 0}$ is an
NP-type sequence, and the assertion (\ref{A-gp0*-NP}) holds.

\textbf{Case 3}: $\delta \left( p_{0}^{\ast }\right) =0$, $\Delta \gamma
_{1}\left( p_{0}^{\ast }\right) >0$ and $\gamma _{1}\left( p_{0}^{\ast
}\right) \leq 0$.

\textbf{Case 3.1}: $\delta \left( p_{0}^{\ast }\right) =0$, $\Delta \gamma
_{1}\left( p_{0}^{\ast }\right) >0$ and $\gamma _{1}\left( p_{0}^{\ast
}\right) <0$. From Lemma \ref{L-sgngn-p} (ii) there is an $n_{1}\geq 1$ such
that $\gamma _{n}\leq 0$ for $1\leq n\leq n_{1}-1$ and $\gamma _{n}>0$ for $%
n\geq n_{1}+1$. Then the relation (\ref{etan-rr-+}) holds. By using the same
technique as Case 2, we can prove that $\left\{ \xi _{n}\right\} _{n\geq 0}$
is an NP-type sequence, and the assertion (\ref{A-gp0*-NP}) holds.

\textbf{Case 3.2}: $\delta \left( p_{0}^{\ast }\right) =0$, $\Delta \gamma
_{1}\left( p_{0}^{\ast }\right) >0$ and $\gamma _{1}\left( p_{0}^{\ast
}\right) =0$. By Lemma \ref{L-sgngn-p} (i) we see that $\gamma _{n}\geq 0$
for all $n\geq 1$. Using the same technique as Case 1, we can prove that $%
\left\{ \xi _{n}\right\} _{n\geq 0}$ is an NP-type sequence, and the
assertion (\ref{A-gp0*-NP}) holds.

\textbf{Case 4}: $\delta \left( p_{0}^{\ast }\right) >0$, $\Delta \gamma
_{1}\left( p_{0}^{\ast }\right) <0$ and $\gamma _{1}\left( p_{0}^{\ast
}\right) \leq 0$. From Lemma \ref{L-sgngn-p} (ii) there is an $n_{1}\geq 1$
such that $\gamma _{n}\leq 0$ for $1\leq n\leq n_{1}-1$ and $\gamma _{n}>0$
for $n\geq n_{1}+1$. By using the same technique as Case 2, we can prove
that $\left\{ \xi _{n}\right\} _{n\geq 0}$ is an NP-type sequence, and the
assertion (\ref{A-gp0*-NP}) holds.

Finally, we prove that $p_{0}^{\ast }$ is the best possible. In fact,
solving the inequality%
\begin{equation*}
\lim_{x\rightarrow 1}\frac{\left[ F\left( -a,b;c;x\right) \right] ^{1/a}}{%
M_{p}\left( 1-x,b/c\right) }\leq 1.
\end{equation*}%
for $p$ gives $p\geq p_{0}^{\ast }$; this together with the increasing
property of $p\mapsto M_{p}\left( 1-x,b/c\right) $ implies that $p_{0}^{\ast
}$ can not be replaced by any smaller number than $p_{0}^{\ast }$. This
completes the proof.
\end{proof}

\section{Proof of Theorem \protect\ref{MT-4}}

\begin{proof}[Proof of Theorem \protect\ref{MT-4}]
As shown in the proof of Theorem \ref{MT-3}, $\xi _{0}\left( p\right) =0$ and%
\begin{equation*}
\xi _{1}\left( p_{0}^{\ast }\right) =\frac{a+c}{c+1}-p_{0}^{\ast
}=p_{0}-p_{0}^{\ast }>0
\end{equation*}%
due to $p_{0}^{\ast }<p_{0}<1$. If we prove that%
\begin{equation}
g_{p_{0}^{\ast }}\left( x\right) =\sum_{n=0}^{\infty }\xi _{n}\left(
p_{0}^{\ast }\right) x^{n}\text{ is an PN-type power series with }%
g_{p_{0}^{\ast }}\left( 1\right) <0,  \label{A-gp0*-PN}
\end{equation}%
then by Proposition \ref{P-sgnS} there is an $x_{0}\in \left( 0,1\right) $
such that $g_{p_{0}^{\ast }}\left( x\right) >0$ for $x\in \left(
0,x_{0}\right) $ and $g_{p_{0}^{\ast }}\left( x\right) <0$ for $x\in \left(
x_{0},1\right) $. This, by (\ref{dfp}), implies that $f_{p_{0}^{\ast
}}\left( x\right) $ is strictly increasing on $\left( 0,x_{0}\right) $ and
decreasing on $\left( x_{0},1\right) $. It then follows that%
\begin{equation*}
>1=\min \left\{ f_{p_{0}^{\ast }}\left( 0\right) ,f_{p_{0}^{\ast }}\left(
1\right) \right\} <f_{p_{0}^{\ast }}\left( x\right) \leq f_{p_{0}^{\ast
}}\left( x_{0}\right) ,
\end{equation*}%
which implies that the inequality (\ref{Hm>Pm-p0*}) holds for $x\in \left(
0,1\right) $.

Note that $\left( a,b,c\right) \in \left( E_{-1,-1,\pm 1}\cup E_{-1,\pm
1,1}\right) \backslash E_{0,0,1}$ can be divided into the following four
cases%
\begin{equation*}
\begin{array}{cccc}
\text{Case 1:} & \delta \left( p_{0}^{\ast }\right) \leq 0, & \Delta \gamma
_{1}\left( p_{0}^{\ast }\right) \leq 0, & \gamma _{1}\left( p_{0}^{\ast
}\right) <0; \\
\text{Case 2:} & \delta \left( p_{0}^{\ast }\right) <0, & \Delta \gamma
_{1}\left( p_{0}^{\ast }\right) \leq 0, & \gamma _{1}\left( p_{0}^{\ast
}\right) \geq 0; \\
\text{Case 3:} & \delta \left( p_{0}^{\ast }\right) =0, & \Delta \gamma
_{1}\left( p_{0}^{\ast }\right) <0, & \gamma _{1}\left( p_{0}^{\ast }\right)
\geq 0; \\
\text{Case 4:} & \delta \left( p_{0}^{\ast }\right) <0 & \Delta \gamma
_{1}\left( p_{0}^{\ast }\right) >0 & \gamma _{1}\left( p_{0}^{\ast }\right)
\geq 0.%
\end{array}%
\end{equation*}%
We now distinguish the above four cases to prove the assertion (\ref%
{A-gp0*-PN}).

\textbf{Case 1}: $\delta \left( p_{0}^{\ast }\right) \leq 0,\Delta \gamma
_{1}\left( p_{0}^{\ast }\right) \leq 0$ and $\gamma _{1}\left( p_{0}^{\ast
}\right) <0$. By Lemma \ref{L-sgngn-p} (i) we see that $\gamma _{n}\leq 0$
for all $n\geq 1$. \textbf{First}, although $\xi _{1}\left( p_{0}^{\ast
}\right) >0$, there must be a negative term in the sequence $\{\xi
_{n}\}_{n\geq 0}$. If not, namely, there is no negative term, then $\xi
_{n}>0$ for all $n\in \mathbb{N}$. This implies $g_{p_{0}^{\ast }}\left(
x\right) >0$ for all $x\in \left( 0,1\right) $, which, by (\ref{dfp}),
yields that $f_{p_{0}^{\ast }}\left( x\right) $ is increasing on $\left(
0,1\right) $. It results in%
\begin{equation}
1=f_{p_{0}^{\ast }}\left( 0\right) <f_{p_{0}^{\ast }}\left( x\right)
<f_{p_{0}^{\ast }}\left( 1\right) =1,  \label{1>1}
\end{equation}%
which leads to a contradiction.

\textbf{Second}, assume that $\xi _{n_{0}}$ ($n_{0}\geq 2$) is the first
negative term of the sequence $\{\xi _{n}\}_{n\geq 0}$, that is, $\xi
_{n}\geq 0$ for $1\leq n\leq n_{0}-1$ and $\xi _{n_{0}}<0$. We claim that
all terms after that ($n\geq n_{0}+1$) are negative, which indicates that $%
\left\{ \xi _{n}\right\} _{n\geq 0}$ is an PN-type sequence. We prove now
that $\xi _{n}<0$ for all $n\geq n_{0}+1$ by induction. Since%
\begin{equation}
\eta _{n_{0}+1}=\xi _{n_{0}+1}-\theta _{n_{0}}\xi _{n_{0}}=\lambda
_{n_{0}}\left( \xi _{n_{0}}+\theta _{n_{0}-1}\left( -\xi _{n_{0}-1}\right)
\right) +\tau _{n_{0}}\gamma _{n_{0}}u_{n_{0}-1}<0,  \label{etan0+1<0}
\end{equation}%
we have $\xi _{n_{0}+1}<\theta _{n_{0}}\xi _{n_{0}}<0$.

Suppose that $\xi _{k}\leq 0$ for $n_{0}+1\leq k\leq m$. We show that $\xi
_{k+1}\leq 0$ by induction. The recurrence relation (\ref{xin-rr-2}) yields%
\begin{equation}
\eta _{k+1}-\lambda _{k}\eta _{k}=\tau _{k}\gamma _{k}u_{k-1}\leq 0
\label{etan+1<}
\end{equation}%
for $k\geq n_{0}+1$. Taking into account (\ref{etan0+1<0}) and (\ref{etan+1<}%
) we obtain%
\begin{equation*}
\eta _{k+1}<\lambda _{k}\eta _{k}<\left( \lambda _{k}\lambda _{k-1}\right)
\eta _{k-1}<\cdot \cdot \cdot <\left( \lambda _{k}\lambda _{k-1}\cdot \cdot
\cdot \lambda _{n_{0}+1}\right) \eta _{n_{0}+1}<0,
\end{equation*}%
which yields $\eta _{k+1}=\xi _{k+1}-\theta _{k}\xi _{k}<0$. Using the
induction hypothesis we immediately deduce that $\xi _{k+1}<\theta _{k}\xi
_{k}<0$. By induction we acquire that $\xi _{n}<0$ for all $n\geq n_{0}$.

\textbf{Third}, we claim that $g_{p_{0}^{\ast }}\left( 1\right) <0$. If not,
that is, $g_{p_{0}^{\ast }}\left( 1\right) \geq 0$, then by Proposition \ref%
{P-sgnS}, $g_{p_{0}^{\ast }}\left( x\right) \geq 0$ for all $x\in \left(
0,1\right) $. This also leads to the same contradiction as (\ref{1>1}).
Consequently, $\left\{ \xi _{n}\right\} _{n\geq 0}$ is a PN-type sequence,
and the assertion (\ref{A-gp0*-PN}) is valid.

\textbf{Case 2:} $\delta \left( p_{0}^{\ast }\right) <0$, $\Delta \gamma
_{1}\left( p_{0}^{\ast }\right) \leq 0$ and $\gamma _{1}\left( p_{0}^{\ast
}\right) \geq 0$. By Lemma \ref{L-sgngn-p} (ii) there is an $n_{1}\geq 1$
such that $\gamma _{n}>0$ for $1\leq n\leq n_{1}-1$ and $\gamma _{n}<0$ for $%
n\geq n_{1}+1$. Then Lemma \ref{L-xin-rr-a1} reveals%
\begin{equation}
\eta _{n+1}-\lambda _{n}\eta _{n}=\tau _{n}\gamma _{n}u_{n-1}\left\{
\begin{array}{ll}
>0 & \text{if }1\leq n\leq n_{1}-1, \\
<0 & \text{if }n\geq n_{1}+1,%
\end{array}%
\right.  \label{etan-rr+-}
\end{equation}%
where $\theta _{n},\tau _{n},\gamma _{n}>0$ for $n\geq 1$ and $\lambda
_{n}>0 $ for $n\geq 2$. We first show that $\xi _{n}=\xi _{n}\left(
p_{0}^{\ast }\right) >0$ for $1\leq n\leq n_{1}$ by induction. Since $\xi
_{0}=0$, $\xi _{1}=p_{0}-p_{0}^{\ast }>0$, we have%
\begin{equation*}
\eta _{1}=\xi _{1}-\theta _{0}\xi _{0}=\xi _{1}>0\text{.}
\end{equation*}%
Moreover, from Lemma \ref{L-l1p0*>0} we see that%
\begin{equation*}
\lambda _{1}=\lambda _{1}\left( p_{0}^{\ast }\right) =\frac{a+2c+2-\left(
3+c\right) p_{0}^{\ast }}{2\left( c+2\right) }>0\text{.}
\end{equation*}%
Using the first relation of (\ref{etan-rr+-}) for $n=1$ yields%
\begin{equation*}
\eta _{2}=\xi _{2}-\theta _{1}\xi _{1}=\lambda _{1}\eta _{1}+\tau _{1}\gamma
_{1}u_{0}>0,
\end{equation*}%
which implies that $\xi _{2}>\theta _{1}\xi _{1}>0$. Suppose that $\xi
_{k}>0 $ for $1\leq k\leq n\leq n_{1}-1$. An application of the first
relation of (\ref{etan-rr+-}) gives%
\begin{equation*}
\eta _{k+1}>\lambda _{k}\eta _{k}>\lambda _{k}\lambda _{k-1}\eta
_{k-1}>\cdot \cdot \cdot >\left( \lambda _{k}\lambda _{k-1}\cdot \cdot \cdot
\lambda _{2}\right) \eta _{2}>0,
\end{equation*}%
which implies that $\xi _{k+1}-\theta _{k}\xi _{k}>0$, and therefore, $\xi
_{k+1}>\theta _{k}\xi _{k}>0$ by the induction hypothesis. It follows from
induction that $\xi _{n}>0$ for $1\leq n\leq n_{1}$.

Second, there must be a negative term in the sequence $\left\{ \xi
_{n}\right\} _{n\geq n_{1}+1}$. If not, namely, there is no negative term,
then $\xi _{n}>0$ for all $n\geq n_{1}+1$. Thus, $\xi _{n}>0$ for $n\in
\mathbb{N}$, and then, $g_{p_{0}^{\ast }}\left( x\right) =\sum_{n=0}^{\infty
}\xi _{n}x^{n}>0$ for all $x\in \left( 0,1\right) $, which, by (\ref{dfp}),
yields that $f_{p_{0}^{\ast }}\left( x\right) $ is increasing on $\left(
0,1\right) $. It leads to the same contradiction as (\ref{1>1}). Now assume
that $\xi _{n_{0}}$ ($n_{0}\geq n_{1}+1$) is the first negative term of the
sequence $\{\xi _{n}\}_{n\geq n_{1}+1}$, that is, $\xi _{n}\geq 0$ for $%
n_{1}+1\leq n\leq n_{0}-1$ and $\xi _{n_{0}}<0$. We claim that all terms
after that ($n\geq n_{0}+1$) are negative, which indicates that $\left\{ \xi
_{n}\right\} _{n\geq 0}$ is a PN-type sequence. We now prove that $\xi
_{n}<0 $ for all $n\geq n_{0}+1\geq n_{1}+2$ by induction. Using the second
relation of (\ref{etan-rr+-}) yields%
\begin{equation*}
\eta _{n_{0}+1}=\xi _{n_{0}+1}-\theta _{n_{0}}\xi _{n_{0}}=\lambda
_{n_{0}}\left( \xi _{n_{0}}+\theta _{n_{0}-1}\left( -\xi _{n_{0}-1}\right)
\right) +\tau _{n_{0}}\gamma _{n_{0}}u_{n_{0}-1}<0,
\end{equation*}%
which gives that $\xi _{n_{0}+1}<\theta _{n_{0}}\xi _{n_{0}}<0$. Suppose
that $\xi _{k}<0$ for $n_{0}+1\leq k\leq n$. Using the second relation of (%
\ref{etan-rr+-})\ again yields%
\begin{equation*}
\eta _{k+1}<\lambda _{k}\eta _{k}<\lambda _{k}\lambda _{k-1}\eta
_{k-1}<\cdot \cdot \cdot <\left( \lambda _{k}\lambda _{k-1}\cdot \cdot \cdot
\lambda _{n_{0}+1}\right) \eta _{n_{0}+1}<0,
\end{equation*}%
which leads to $\xi _{k+1}-\theta _{k}\xi _{k}<0$, and therefore, $\xi
_{k+1}<\theta _{k}\xi _{k}<0$. By induction we arrive at $\xi _{k}<0$ for
all $k\geq n_{0}$. Consequently, $\left\{ \xi _{n}\right\} _{n\geq 0}$ is a
PN-type sequence, and the assertion (\ref{A-gp0*-PN}) holds.

\textbf{Case 3}: $\delta \left( p_{0}^{\ast }\right) =0$, $\Delta \gamma
_{1}\left( p_{0}^{\ast }\right) <0$ and $\gamma _{1}\left( p_{0}^{\ast
}\right) \geq 0$.

\textbf{Case 3.1}: $\delta \left( p_{0}^{\ast }\right) =0$, $\Delta \gamma
_{1}\left( p_{0}^{\ast }\right) <0$ and $\gamma _{1}\left( p_{0}^{\ast
}\right) >0$. From Lemma \ref{L-sgngn-p} (ii) there is an $n_{1}\geq 1$ such
that $\gamma _{n}>0$ for $1\leq n\leq n_{1}-1$ and $\gamma _{n}<0$ for $%
n\geq n_{1}+1$. Then the relation (\ref{etan-rr+-}) holds. By using the same
technique as Case 2, we can prove that $\left\{ \xi _{n}\right\} _{n\geq 0}$
is a PN-type sequence, and the assertion (\ref{A-gp0*-PN}) holds.

\textbf{Case 3.2}: $\delta \left( p_{0}^{\ast }\right) =0$, $\Delta \gamma
_{1}\left( p_{0}^{\ast }\right) <0$ and $\gamma _{1}\left( p_{0}^{\ast
}\right) =0$. By Lemma \ref{L-sgngn-p} (i) we see that $\gamma _{n}\leq 0$
for all $n\geq 1$. Using the same technique as Case 1, we can prove that $%
\left\{ \xi _{n}\right\} _{n\geq 0}$ is a PN-type sequence, and the
assertion (\ref{A-gp0*-PN}) holds.

\textbf{Case 4}: $\delta \left( p_{0}^{\ast }\right) <0$, $\Delta \gamma
_{1}\left( p_{0}^{\ast }\right) >0$ and $\gamma _{1}\left( p_{0}^{\ast
}\right) \geq 0$. From Lemma \ref{L-sgngn-p} (ii) there is an $n_{1}\geq 1$
such that $\gamma _{n}>0$ for $1\leq n\leq n_{1}-1$ and $\gamma _{n}<0$ for $%
n\geq n_{1}+1$. By using the same technique as Case 2, we can prove that $%
\left\{ \xi _{n}\right\} _{n\geq 0}$ is a PN-type sequence, and the
assertion (\ref{A-gp0*-PN}) holds.

Finally, we prove that $p_{0}^{\ast }$ is the best possible. In fact,
solving the inequality%
\begin{equation*}
\lim_{x\rightarrow 1}\frac{\left[ F\left( -a,b;c;x\right) \right] ^{1/a}}{%
M_{p}\left( 1-x,b/c\right) }\geq 1.
\end{equation*}%
for $p$ gives $p\leq p_{0}^{\ast }$; this together with the increasing
property of $p\mapsto M_{p}\left( 1-x,b/c\right) $ implies that $p_{0}^{\ast
}$ can not be replaced by any larger number than $p_{0}^{\ast }$. This
completes the proof.
\end{proof}

\section{Conclusions}

In this paper, by using the signs rule of the NP or PN-type power series, we
found the second type of power mean bounds for the Gaussian hypergeometric
mean (sharp at $x=1$). Finally, we present several remarks.

\begin{remark}
From Remark \ref{R-C3,4}, our Theorems \ref{MT-3} and \ref{MT-4} may be
stronger than the result suggested in Conjecture \ref{Cj-BRT}.
\end{remark}

\begin{remark}
We believe that Conjecture \ref{Cj-BRT} is valid, which can be proved by our
Theorems \ref{MT-3} and \ref{MT-4}.
\end{remark}

\begin{remark}
Theorems \ref{MT-3} and \ref{MT-4} only provide sufficient conditions for
the second Hm-Pm inequalities. It is interesting and challenging to find the
necessary and sufficient conditions.
\end{remark}

\begin{remark}
It should be emphasize that the signs rule for the NP-type or PN-type power
series, Proposition \ref{P-sgnS}, play an important role in the proofs of
Theorems \ref{MT-3} and \ref{MT-4}. From the published literatures, this
signs rule is an efficient tool to deal with certain special functions.
\end{remark}


\begin{thebibliography}{99}
\bibitem{Yang-arXiv:2610.10581} Z.-H. Yang and J.-F. Tian, The first type of
power mean bounds for the Gaussian hypergeometric function, arXiv:2610.10581
[math.GM]. https://doi.org/10.48550/arXiv.2610.10581

\bibitem{Carlson-PAMS-16-1965} B. C. Carlson, A hypergeometric mean value,
\emph{Proc. Amer. Math. Soc.} \textbf{16} (1965), no. 4, 759--766.

\bibitem{Carlson-JMAA-7-1963} B. C. Carlson, Lauricella'sh hypergeometric
function $F_{D}$, \emph{J. Math. Anal. Appl.} \textbf{7} (1963), 452--470.

\bibitem{Carlson-PAMS-16-1966} B. C. Carlson, Some inequalities for
hypergeometric functions, \emph{Proc. Amer. Math. Soc.} \textbf{16} (1966),
32--39.

\bibitem{Barnard-CMFT-1-2001} R. W. Barnard and K. C. Richards, A note on
the hypergeometric mean value, \emph{Comput. Meth. Func. Theory} \textbf{1}
(2001), no. 1, 81--88.

\bibitem{Richards-JMAA-308-2005} K. C. Richards, Sharp power mean bounds for
the Gaussian hypergeometric function, \emph{J. Math. Anal. Appl.} \textbf{308%
} (2005), 303--313. 

\bibitem{Barnard-JIPAM-8-2007} R. W. Barnard and K. C. Richards, On
inequalities for hypergeometric analogues of the arithmetic-geometric mean,
\emph{J. Inequal. Pure Appl. Math.}, \textbf{8} (2007), no. 3, Art. 65.

\bibitem{Barnard-JMI-4-2010} R. W. Barnard, K. C. Richards and H. C.
Tiedeman, A survey of some bounds for Gauss' hypergeometric function and
related to mean, \emph{J. Math. Inequal.} \textbf{4} (2010), no. 1, 45--52.

\bibitem{Pales-JMAA-131-1988} Z. P\'{a}les, Inequalities for differences of
powers, \emph{J. Math. Anal. App.} \textbf{131} (1988), 271--281.

\bibitem{Pittinger-UBPEFSMF-680-1980} A. O. Pittinger, Inequalities between
arithmetic and logarithmic means, \emph{Univ. Beogard Publ. Elektr. Fak.
Ser. Mat. Fiz} \textbf{680} (1980), 15--18.

\bibitem{Stolarsky-AMM-87-1980} K. B. Stolarsky, The power and generalized
logarithmic means, \emph{Amer. Math. Monthly.} \textbf{87} (1980), 545--548.
https://doi.org/10.1080/00029890.1980.11995086

\bibitem{Yang-MIA-10-2007} Z.-H. Yang, On the log-convexity of two-parameter
homogeneous functions, \emph{Math. Inequal. Appl.} \textbf{10 }(2007), no.
3, 499--516. 

\bibitem{Yang-JMI-7-2013} Z.-H. Yang, Refinements of a two-sided inequality
for trigonometric functions, Journal of Mathematical Inequalities, vol. 7,
no. 4, pp. 601--615, 2013. doi:10.7153/jmi-07-57

\bibitem{Takeuchi-KMJ-39-2016} S. Takeuchi, A new form of the generalized
complete elliptic integrals, \emph{Kodai Math. J.} \textbf{39} (2016),
202--226. DOI: 10.2996/kmj/1458651700

\bibitem{Zhao-BMSS-46-2023} T. Zhao, Sharp Approximations for Complete
p-Elliptic Integral of the Second Kind byWeighted Power Means, \emph{Bull.
Malays. Math. Sci. Soc.} \textbf{46} (2023): 126.
https://doi.org/10.1007/s40840-023-01523-0

\bibitem{Qiu-JHIEE-3-1997} S.-L. Qiu and J.-M. Shen, On two problems
concerning means, \emph{J. Hangzhou Inst. Electr. Eng.} \textbf{17} (1997),
no. 3, 1--7. (in Chinese)

\bibitem{Qiu-JHIEE-20(1)-2000} S.-L. Qiu, The Muir mean and the complete
elliptic intergral of the second kind, \emph{J. Hangzhou Inst. Electr. Eng.}
\textbf{20} (2000), no. 1, 28--33. (in Chinese)

\bibitem{Alzer-JCAM-172-2004} H. Alzer and S.-L. Qiu, Monotonicity theorems
and inequalities for the complete elliptic integrals, \emph{J. Comput. Appl.
Math.} \textbf{172} (2004), no. 2, 289--312.

\bibitem{Barnard-JMA-31-2000} R.W. Barnard, K. Pearce, and K.C. Richards, An
inequality involving the generalized hypergeometric function and the arc
length of an ellipse, \emph{SIAM J. Math. Anal.} \textbf{31} (2000),
693--699.

\bibitem{Barnard-JMA-32-2000} R.W. Barnard, K. Pearce, and K.C. Richards, A
monotonicity property involving $_{3}F_{2}$ and comparisons of the classical
approximations of elliptical arc length, \emph{SIAM J. Math. Anal.} \textbf{%
32} (2000), 403--419.

\bibitem{Barnard-JMAA-260-2001} R.W. Barnard, K. Pearce, and L. Schovanec,
Inequalities for the perimeter of an ellipse, \emph{J. Math. Anal. Appl.}
\textbf{260} (2001), 295--306.

\bibitem{Yang-RACSAM-115-2021-8} Z.-H. Yang, J.-F. Tian, Y.-R. Zhu, A sharp
lower bound for the complete elliptic integrals of the first kind, \emph{%
Rev. R. Acad. Cienc. Exactas F\'{\i}s. Nat. Ser. A Mat. RACSAM} \textbf{115}
(2021), no. 1, Paper No. 8, 17 pp.

\bibitem{Yang-arXiv-1705.05703} Z.-H. Yang, J. Tian, Convexity and
monotonicity for the elliptic integrals of the first kind and applications,
arXiv:1705.05703 [math.CA]. 

\bibitem{Yang-AADM-13-2019} Z.-H. Yang and J. Tian, Convexity and
monotonicity for elliptic integrals of the first kind and applications,
\emph{Appl. Anal. Discrete Math.} \textbf{13} (2019), 240--260.

\bibitem{Yang-JMAA-462-2018} Z.-H. Yang, W.-M. Qian, Y.-M. Chu, W. Zhang, On
approximating the arithmetic-geometric mean and complete elliptic integral
of the first kind, \emph{J. Math. Anal. Appl.} \textbf{462} (2018),
1714--1726.

\bibitem{Belzunce-I:M-E-40-2007} F. Belzunce, E. Ortega, and J. M. Ruiz, On
non-monotonic ageing properties from the Laplace transform, with actuarial
applications, \emph{Insurance Math. Econom.} \textbf{40} (2007), 1--14.

\bibitem{Yang-JIA-157-2015} Z.-H. Yang, Y.-M. Chu and X.-H. Zhang, Necessary
and sufficient conditions for functions involving the psi function to be
completely monotonic, \emph{J. Inequal. Appl.} \textbf{2015} (2015): 157.

\bibitem{Yang-MIA-21-2018-2} Z.-H. Yang, W.-M. Qian, Y.-M. Chu, W. Zhang, On
approximating the error function, \emph{Math. Inequal. Appl.} \textbf{21}
(2019), no. 2, 469--479. 

\bibitem{Yang-MIA-21-2018-4} Z.-H. Yang, W.-M. Qian and Y.-M. Chu,
Monotonicity properties and bounds involving the complete elliptic integrals
of the first kind, \emph{Math. Inequal. Appl.} \textbf{21} (2018), no. 4,
1185--1199. 

\bibitem{Yang-JMAA-467-2018} Z.-H. Yang, Sharp approximations for the
complete elliptic integrals of the second kind by one-parameter means, \emph{%
J. Math. Anal. Appl.} \textbf{467} (2018), 446--461.

\bibitem{Yang-AMC-348-2019} Z.-H. Yang, Y.-M. Chu and W. Zhang, High
accuracy asymptotic bounds for the complete elliptic integral of the second
kind, \emph{Appl. Math. Comput.} \textbf{348} (2019), 552--564.

\bibitem{Yang-JMAA-470-2019} Z. Yang and J.-F. Tian, Monotonicity rules for
the ratio of two Laplace transforms with applications, \emph{J. Math. Anal.
Appl.} \textbf{470} (2019), 821--845.

\bibitem{Yang-AAA-2014-702718} Z.-H. Yang, Y.-M. Chu, and X.-J, Tao, A
double inequality for the trigamma function and its applications, \emph{%
Abstr. Appl. Anal. }\textbf{2014} (2014), Art. ID 702718, 9 pages.

\bibitem{Yang-JIA-2017-210} Z.-H. Yang, W.-M. Qian, Y.-M. Chu, W. Zhang, On
rational bounds for the gamma function, \emph{J. Inequal. Appl.} \textbf{2017%
} (2017): 210. 

\bibitem{Yang-JMI-12-2018} Z.-H. Yang and J. Tian, Monotonicity and sharp
inequalities related to gamma function, \emph{J. Math. Inequal.} \textbf{12}
(2018), no. 1, 1--22. 

\bibitem{Yang-RACSAM-115-2021} Z.-H. Yang and J.-F. Tian, Sharp bounds for
the Toader mean in terms of arithmetic and geometric means, \emph{Rev. R.
Acad. Cienc. Exactas F\'{\i}s. Nat. Ser. A Mat. RACSAM} \textbf{115}, 99
(2021). Available online at https://doi.org/10.1007/s13398-021-01040-4.

\bibitem{Zhao-RACSAM-117-2023} T. H. Zhao, M. K. Wang, Sharp bounds for the
lemniscatic mean by the weighted Holder mean, \emph{Rev. R. Acad. Cienc.
Exactas F\'{\i}s. Nat. Ser. A Mat. RACSAM} \textbf{117}, 96 (2023).
https://doi.org/10.1007/s13398-023-01429-3

\bibitem{Yang-arxiv-2204.04709} Z.-H. Yang, Recurrence relations of
coefficients involving hypergeometric function with an application, arXiv:
2204.04709 [math.CA]. 

\bibitem{Abramowitz-HMFFGMT-1972} M. Abramowitz and I. A. Stegun (Eds),
\emph{Handbook of Mathematical Functions with Formulas, Graphs, and
Mathematical Tables}, National Bureau of Standards, Applied Mathematics
Series 55, 10th printing, Dover Publications, New York and Washington, 1972.

\bibitem{Tian-RACSAM-115-2021} J.-F. Tian and Z.-H. Yang, New properties of
the divided difference of psi and polygamma functions, \emph{Rev. R. Acad.
Cienc. Exactas F\'{\i}s. Nat. Ser. A Mat. RACSAM} \textbf{15} (2021): 147.
https://doi.org/10.1007/s13398-021-01084-6

\bibitem{Yang-JMAA-455-2017} Z.-H. Yang, Some properties of the divided
difference of psi and polygamma functions, \emph{J. Math. Anal. Appl.}
\textbf{455} (2017), 761--777. https://doi.org/10.1016/j.jmaa.2017.05.081

\bibitem{Yang-IJPAM-2023} Z.-H. Yang and J.-F. Tian, The signs rule for the
Laplace integrals with applications, \emph{Indian J. Pure Appl. Math.}
https://doi.org/10.1007/s13226-023-00447-6.

\bibitem{Qi-RM-76-2021} F. Qi, Lower bound of sectional curvature of
Fisher--Rao manifold of beta distributions and complete monotonicity of
functions involving polygamma functions, \emph{Results Math.} \textbf{76}
(2021), Article 217, 16\symbol{126}pages.
https://doi.org/10.1007/s00025-021-01530-2
\end{thebibliography}
\end{document}